\documentclass[11pt,a4paper,reqno]{amsart}

\usepackage[a4paper,lmargin=2.5cm,rmargin=2cm,tmargin=4cm,bmargin=4cm]{geometry}

\makeatletter
\g@addto@macro\bfseries{\boldmath} 
\makeatother

\usepackage[centertags]{amsmath}
\usepackage{amsfonts}
\usepackage{amssymb}
\usepackage{amsthm}

\usepackage{hyperref}
	\hypersetup{breaklinks=true,colorlinks=true,
linkcolor=MidnightBlue,citecolor=MidnightBlue,
urlcolor=MidnightBlue}

\usepackage{tikz}
\usepackage{tikz-cd}
\usepackage[normalem]{ulem}
\usepackage[shortlabels]{enumitem}

\usepackage[dvipsnames]{xcolor}

\usepackage{mathtools,dsfont}

\usepackage{mathrsfs}
\usepackage{graphicx}

\usepackage{orcidlink}

\newtheorem{theorem}{Theorem}[section]
\newtheorem{introthm}{Theorem}

\newtheorem{lemma}[theorem]{Lemma}

\newtheorem{proposition}[theorem]{Proposition}

\newtheorem{corollary}[theorem]{Corollary}

\newtheorem{introproblem}{Problem}

\theoremstyle{definition}
\newtheorem{definition}[theorem]{Definition}
\newtheorem{example}[theorem]{Example}
\newtheorem*{examplenonumber}{Example}

\theoremstyle{remark}
\newtheorem{remark}[theorem]{Remark}

\newtheorem{fact}[theorem]{Fact}

\numberwithin{equation}{section}

\newcommand{\R}{\mathbb{R}}

\newcommand{\N}{\mathbb{N}}

\newcommand{\vertiii}[1]{{\left\vert\kern-0.25ex\left\vert\kern-0.25ex\left\vert #1 
		\right\vert\kern-0.25ex\right\vert\kern-0.25ex\right\vert}}

\makeatletter
\newcommand{\markthis}[3]{% #1 = marker, #2 = label, #3 = relation
	\overset{% the marker and the label
		\textup{\makebox[0pt]{#1}}%
		\def\@currentlabel{#1}%
		\ltx@label{#2}%
	}{% the relation
		#3%
	}%
}

\DeclareMathOperator{\co}{co}

\DeclareMathOperator{\id}{Id}

\DeclareMathOperator{\sign}{sign}

\newcommand{\nn}[1]{{\left\vert\kern-0.25ex\left\vert\kern-0.25ex\left\vert #1 
		\right\vert\kern-0.25ex\right\vert\kern-0.25ex\right\vert}}
\renewcommand{\geq}{\geqslant}
\renewcommand{\leq}{\leqslant}

\newcommand{\restricted}{\mathord{\upharpoonright}}

\newcommand{\NA}{\operatorname{NA}}
\newcommand{\QNA}{\operatorname{QNA}}
\newcommand{\spann}{\operatorname{span}}

\newcommand{\supp}{\operatorname{supp}}

\newcommand{\wot}{\operatorname{WOT}}

\newcommand{\eps}{\varepsilon}

\newcommand{\weakstar}{\mathcal{L}_{w^*, w^*}(Y^*, X^*)}
\newcommand{\weakstarfiniterank}{\mathcal{F}_{w^*, w^*}(Y^*, X^*)}
\newcommand{\propertystar}{(P^\star)}
\newcommand{\propertystarstar}{(P^{\star\star})}

\newcounter{smallromans}

	{\end{list}}

\begin{document}

	\title[On operators whose adjoints or second adjoints attain their norms]{On operators whose adjoints or second adjoints attain their norms}

\date{\today}

\author[Dantas]{Sheldon Dantas\orcidlink{0000-0001-8117-3760}}
\address[Dantas]{Czech Technical University in Prague, FEE, Department of Mathematics, Technick\'a 2, 16627, Prague 6, Czech Republic. \newline
\href{https://orcid.org/0000-0001-8117-3760}{ORCID: \texttt{0000-0001-8117-3760}}}
\email{\texttt{sheldon.dantas@fel.cvut.cz}}
\urladdr{www.sheldondantas.com}

	\author[Jung]{Mingu Jung\orcidlink{0000-0003-2240-2855}}
\address[Jung]{Department of Mathematics \& Research Institute for Natural Sciences, Hanyang University, 04763 Seoul, Republic of Korea. \newline
\href{http://orcid.org/0000-0003-2240-2855}{ORCID: \texttt{0000-0003-2240-2855} }}
\email{mingujung@hanyang.ac.kr}

 	\author[Martín]{Miguel Martín\orcidlink{0000-0003-4502-798X}}
 \address[Martín]{Department of Mathematical Analysis and Institute of Mathematics (IMAG), University of Granada, E-18071 Granada, Spain. \newline
 	\href{https://orcid.org/0000-0003-4502-798X}{ORCID: \texttt{0000-0003-4502-798X}}}
 \email{\texttt{mmartins@ugr.es}}
  \urladdr{https://www.ugr.es/local/mmartins/}

	\begin{abstract} 
A long-standing open problem asks whether there exists an infinite-dimensional Banach space on which every bounded linear operator attains its norm. With $\mathrm{NA}_1(X,Y)$ and $\mathrm{NA}_2(X,Y)$ denoting the classes of operators whose adjoints and second adjoints, respectively, attain their norms, we prove that
 \[ \mathrm{NA}_2(c_0,c_0)=\mathcal{L}(c_0,c_0) \qquad\text{and}\qquad \mathrm{NA}_2(\ell_1,\ell_1)=\mathcal{L}(\ell_1,\ell_1), \]
providing, to the best of our knowledge, the first known infinite-dimensional spaces on which every operator has a norm-attaining second adjoint. 
Building on the result for $c_0$, we undertake a systematic study of this equality within a natural family of $\ell_1$-preduals given by hyperplanes of $c$, obtaining a complete characterization in this setting. In particular, we prove that \[ \mathrm{NA}_2(c,c)\neq \mathcal{L}(c,c), \qquad\text{whereas}\qquad \mathrm{NA}_3(c,c)=\mathcal{L}(c,c). \]  
We also establish Holub--Mujica-type theorems for the classes $\mathrm{NA}_1$ and $\mathrm{NA}_2$. More precisely, under suitable separability and approximation property assumptions, the identity $\mathcal L(X,Y)=\mathrm{NA}_1(X,Y)$ forces every operator from $X$ into $Y$ to be compact, whereas $\mathcal L(X,Y)=\mathrm{NA}_2(X,Y)$ forces every weakly compact operator from $X$ into $Y$ to be compact. Finally, strengthening a construction of Ostrovskii, we show that every infinite-dimensional Banach space admits an equivalent norm and a projection whose second adjoint does not attain its norm. 
	\end{abstract}

	
	\maketitle

\section{Introduction}

A classical theorem of James characterizes reflexive Banach spaces as those on which every continuous linear functional attains its norm. For operators between infinite-dimensional Banach spaces, the corresponding phenomenon is much less understood. In particular, the following long-standing problem was explicitly posed by Ostrovskii in \cite[\textsection12, p.~65]{MP}; see also \cite[Problem 8]{KOS} and \cite[Problem 217]{GMZ}.

Throughout the paper, all Banach spaces are assumed to be \emph{real}. Given Banach spaces $X$ and $Y$, we denote by $\mathcal{L}(X,Y)$ the space of all bounded linear operators from $X$ into $Y$. An operator $T \in \mathcal{L}(X,Y)$ is said to \emph{attain its norm} if there exists a norm-one element $x \in X$ such that $\|Tx\|=\|T\|$. We denote by $\NA(X,Y)$ the set of all such operators. 

\begin{introproblem} \label{Q1} Is there an infinite-dimensional Banach space $X$ such that $\mathcal{L}(X,X) = \NA(X, X)$?
\end{introproblem}

This problem is related to the question of whether $\mathcal{L}(X,X)$ can be reflexive for some infinite-dimensional Banach space $X$ \cite{Godefroy-Saphar,Heinrich,Ruckle}. A standard application of James' theorem shows that a positive solution to Problem~\ref{Q1} would force $\mathcal{L}(X,X)$ to be reflexive. Consequently, $X$ would be reflexive and, by a result of Kalton \cite{Kalton}, separable. Results of Holub \cite{H} and Mujica \cite{Mujica} impose a further strong restriction: such a space cannot have the (compact) approximation property. These results belong to a broader circle of connections among norm attainment, compactness of operators, and reflexivity of operator spaces; we recall the relevant formulation in Theorem~\ref{thm:DJM}.

Norm-attaining operators have been studied extensively, but much less is known about norm attainment after passing to adjoints. As Problem \ref{Q1} remains open, it is natural to ask whether its adjoint-level analogues are more tractable. As we shall see, these variants retain substantial geometric content, but their behavior differs sharply from that of the original problem.

\begin{introproblem} \label{Q2} 
Is there an infinite-dimensional Banach space $X$ such that $T^* \in \NA(X^*, X^*)$ for every operator $T \in \mathcal{L}(X,X)$?
\end{introproblem}

\begin{introproblem} \label{Q3} Is there an infinite-dimensional Banach space $X$ such that $T^{**} \in \NA(X^{**}, X^{**})$ for every operator $T \in \mathcal{L}(X,X)$?
\end{introproblem}

To study these problems systematically, we introduce the sets
\begin{equation*}
	\NA_1(X, Y) = \{T \in \mathcal{L}(X, Y)\colon T^* \in \NA(Y^*, X^*)\}  
\end{equation*}
and 
\begin{equation*}
	\NA_2(X, Y) = \{T \in \mathcal{L}(X,Y)\colon T^{**} \in \NA(X^{**}, Y^{**})\},
\end{equation*}
so that Problems \ref{Q2} and \ref{Q3} ask, respectively, whether the equalities $\mathcal{L}(X,X) =\NA_1 (X,X)$ and $\mathcal{L}(X,X)=\NA_2 (X,X)$ may hold for some infinite-dimensional $X$. Analogously, we may define $\NA_k(X,Y)$ for every $k\geq 3$. By a straightforward application of the Hahn-Banach theorem, 
\[
\NA(X,Y) \subseteq \NA_1 (X, Y) \subseteq \NA_2 (X,Y).
\]
Neither inclusion is reversible in general; we discuss this in Section~\ref{sec:preliminaries}. The approximation-theoretic picture is nevertheless quite different. In his foundational paper \cite{Lindenstrauss}, Lindenstrauss showed both that $\NA(X,Y)$ need not be norm dense in $\mathcal{L}(X,Y)$ and that $\NA_2(X,Y)$ is norm dense in $\mathcal{L}(X,Y)$ for arbitrary Banach spaces $X$ and $Y$. The former failure may persist even for compact operators \cite{Martin2014}. Zizler subsequently strengthened the second-adjoint density theorem by proving that $\NA_1(X,Y)$ is norm dense in $\mathcal{L}(X,Y)$ for arbitrary Banach spaces $X$ and $Y$ \cite{Zizler}. Thus, although $\NA_1(X,Y)$ and $\NA_2(X,Y)$ may be proper subsets of $\mathcal{L}(X,Y)$, they are always norm dense. The present paper addresses the substantially stronger question of when every operator, rather than merely an arbitrarily small perturbation of it, has a norm-attaining adjoint or second adjoint. 

We refer the reader to \cite{Acostasurvey,JungMartinRueda2023} and the references therein for further results on the denseness of norm-attaining operators and its geometric consequences.

Our results show that this stronger universal phenomenon is both possible and highly rigid. Universal second-adjoint norm attainment occurs on the classical non-reflexive spaces $c_0$ and $\ell_1$, but within a natural family of $\ell_1$-preduals, it admits a sharp classification. Under approximation hypotheses, universal norm attainment at the first or second adjoint level forces compactness phenomena, while an equivalent renorming can always destroy the second-adjoint property. 

\subsection{Main results and organization}
Section~\ref{sec:preliminaries} is devoted to general results on
the classes $\QNA$, $\NA_1$, and $\NA_2$. In particular, we show that
non-reflexivity of the range space forces the inclusions between these
classes to be strict in suitable settings (Theorem~\ref{thm:NA1_NA2_converse}), yielding a characterization of
reflexivity of the range space (Corollary~\ref{cor:charac-Y-reflexive}). We also characterize the Schur property of the dual of the domain space (Corollary~\ref{cor:Schur}).

\subsubsection{Main positive results for second adjoints}
In Section~\ref{sec:ell_1_preduals} we give a positive answer to Problem~\ref{Q3} for the two classical non-reflexive spaces $c_0$ and $\ell_1$,  providing, to the best of our knowledge, the first infinite-dimensional examples of Banach spaces on which every operator has a norm-attaining second adjoint.

\begin{introthm}[\mbox{\textrm Theorems \ref{theorem:c0} and \ref{thm:NA2_l1_l1}}]\label{thm:main_c0}
    The following equalities hold: 
    \[
    \mathcal{L}(c_0,c_0) = \NA_2 (c_0,c_0) \quad\text{and}\quad \mathcal{L}(\ell_1,\ell_1)=\NA_2(\ell_1,\ell_1). 
    \]
\end{introthm}

Moreover, the equality also holds for $c_0$-sums of finite-dimensional spaces (Remark~\ref{rem:c0_sum_finite_dim}), for spaces $c_0(\Gamma)$ (Remark~\ref{rem-c0gamma}) and  $\ell_1(\Gamma)$ (Remark~\ref{rem-ell1gamma}). These conclusions are specific to norm-attainment at the level of the second adjoint. Indeed, Example \ref{ex:strict_inclusion} and Proposition \ref{prop:l1_l1} show that 
\[
    \mathcal{L}(c_0,c_0) \neq \NA_1 (c_0,c_0) \quad\text{and}\quad \mathcal{L}(\ell_1,\ell_1) \neq \NA_1(\ell_1,\ell_1). 
\]

Both proofs are driven by disjoint-support phenomena in $\ell_1$, but in different ways. For operators on $c_0$, we apply a disjoint-support perturbation argument in $\ell_1$ (Lemma \ref{lemma:disjoint_support_lemma} and Corollary \ref{cor:finite-head-norming}). For operators on $\ell_1$, we begin with a maximizing sequence, separate its weak$^*$ limit from an asymptotically disjoint remainder, and exploit the decomposition $\ell_\infty^*= \ell_1\oplus_1 c_0^\perp$ which follows from the fact that $c_0$ is an $M$-ideal in $\ell_\infty$ (Fact~\ref{fact:Hewitt-Yosida}). 

Next, building on the $c_0$ part of Theorem~\ref{thm:main_c0}, we carry out a systematic study of the equality $\mathcal{L}(W_\alpha,W_\beta)=\NA_2 (W_\alpha, W_\beta)$ for the $\ell_1$-predual spaces $W_\alpha$ and $W_\beta$ (see Definition~\ref{def:Walpha} for the definition of these spaces).

\begin{introthm}[\mbox{\textrm Theorems \ref{theorem:c0}, \ref{full-characterization_target_c0}, \ref{thm:target_c}, and \ref{thm:target_W_beta}}]\label{thm:main_W_alpha_W_beta}
    Let $\alpha,\beta\in S_{\ell_1}$ such that both $W_\alpha^*$ and $W_\beta^*$ are isometric to $\ell_1$. 
  \begin{enumerate}[label=(\roman*)]
  \itemsep0.25em
  \item If $W_\alpha$ is isometric to $c_0$ or $c$, then $\mathcal{L}(W_\alpha, c_0)=\NA_2 (W_\alpha, c_0) $. 
  \item If $W_\alpha$ is isometric to neither $c_0$ nor $c$, then $\mathcal{L}(W_\alpha, c_0)=\NA_2 (W_\alpha, c_0)$ if and only if $\alpha$ is finitely supported. 
  \item $\mathcal{L} (W_\alpha, c) \neq \NA_2 (W_\alpha, c)$.
  \item If $W_\beta$ is isometric to neither $c_0$ nor $c$, then $\mathcal{L}(W_\alpha,W_\beta) \neq \NA_2 (W_\alpha, W_\beta)$. 
    \end{enumerate}
\end{introthm}

In particular, Theorem \ref{thm:main_W_alpha_W_beta} provides a complete characterization of the equality $\mathcal{L}(W_\alpha,W_\beta)=\NA_2 (W_\alpha,W_\beta)$ for all pairs of $\ell_1$-preduals within the class $\{W_\alpha\}$. As a byproduct, we show that the Schur property of $X^*$ does not suffice to guarantee $\mathcal{L}(X,c_0) = \NA_2 (X, c_0)$ (Example~\ref{exa:the_converse_of_Schur_result}). 

As a further consequence of Theorems ~\ref{thm:main_c0} and \ref{thm:main_W_alpha_W_beta}, we obtain a strict separation between two consecutive adjoint levels: taking $X=c$, we have  

\begin{examplenonumber}[\mbox{\textrm Corollary \ref{cor:ccN3-noN2}}]
There exists a Banach space $X$ such that $\NA_3(X,X)=\mathcal L (X,X)$ while $\NA_2(X,X) \neq \mathcal L (X,X)$.
\end{examplenonumber}

\subsubsection{Holub--Mujica-type theorems for \texorpdfstring{$\NA_1$}{NA1} and \texorpdfstring{$\NA_2$}{NA2}} 
Section~\ref{section:HolubMujica} deals with the natural question of whether the classical Holub--Mujica theorem
admits analogues when norm attainment is replaced by the weaker
properties defined using $\NA_1$ and $\NA_2$.

The natural ambient operator space changes with the adjoint level. In the $\NA_1$ case, we work with the subspace $\mathcal{L}_{w^*,w^*}(Y^*, X^*)$ of $\mathcal{L}(Y^*,X^*)$ of all weak$^*$-to-weak$^*$ continuous operators in place of $\mathcal{L}(X,Y)$, whereas in the $\NA_2$ case we work with the corresponding subspace of  $\mathcal{L}(X^{**},Y^{**})$ of second adjoints. We introduce the appropriate weak$^*$-analogues of the notions used in \cite{DJM}, while versions of the Principle of Local Reflexivity (Lemma \ref{lem:PLR_Johnson} and Theorem \ref{thm:PLR}) play a crucial role by providing the key approximation step. 

Our main result in this direction is the following.

\begin{introthm}[\mbox{\textrm Theorems \ref{theorem-holub-mujica-NA1} and \ref{Holub-for-NA2}}]\label{thm:main_holub_mujica}
    Let $X$ and $Y$ be Banach spaces.
      \begin{enumerate}[label=(\roman*)]
  \itemsep0.25em
  \item Suppose that $Y^*$ is separable and that the pair $(Y^*,X^*)$
has the bounded approximation property. If $\mathcal{L}(X,Y)=\NA_1 (X,Y)$, then every bounded linear operator from $X$ into $Y$ is compact. 
  \item Suppose that $X^{**}$ is separable and that the pair $(X^{**},Y)$
has the bounded approximation property. If $\mathcal{L}(X,Y)=\NA_2 (X,Y)$, then every weakly compact operator from $X$ into $Y$ is compact. 
  \end{enumerate}
\end{introthm}

Thus, under these hypotheses, universal first-adjoint norm attainment forces compactness of all operators, while universal second-adjoint norm attainment controls precisely the weakly compact part of $\mathcal L(X,Y)$. Observe that when $X$ is reflexive, the above results recover the extension of the Holub--Mujica theorem provided in \cite{DJM} (Theorem~\ref{thm:DJM} below).

As a consequence, we obtain that there is no infinite-dimensional Banach space $X$ with $X^{**}$ separable such that $\NA_2(X,c_0)=\mathcal{L}(X,c_0)$ (Proposition~\ref{prop:NA2_when_target_is_c0}).

\subsubsection{Renormings with non-norm-attaining second adjoints}

Building on a construction due to Ostrovskii \cite{Ostrovskii}, which yields, on every infinite-dimensional Banach space, an equivalent norm under which a certain operator fails to attain its norm, we establish in Section~\ref{sec:Ostrovskii} the following stronger result.

\begin{introthm}[\mbox{\textrm Theorem \ref{Ostrovskii-for-NA1-and-NA2}}]\label{thm:main_renorming}
 Let $X$ be an infinite-dimensional Banach space. Then there exist an equivalent norm $\vertiii{\cdot}$ on $X$ and a projection $P$ on $(X, \vertiii{\cdot})$ such that $P \not\in \NA_2((X, \vertiii{\cdot}), (X, \vertiii{\cdot}))$.
\end{introthm}

In particular, the identity $\mathcal{L}(X,X)=\NA_2 (X,X)$ is not stable under equivalent renormings.

\subsection{Further notation}\label{subsec:notation}
Given a Banach space $X$, we denote by $B_X$ and $S_X$ its closed unit ball and unit sphere, respectively. For Banach spaces $X$ and $Y$, we write $X \equiv Y$ when $X$ and $Y$ are linearly isometric. The symbols $\mathcal{K}(X,Y), \mathcal{W}(X,Y)$, and $\mathcal{F}(X,Y)$ stand for the subspaces of $\mathcal{L}(X,Y)$ consisting of all compact, weakly compact, and finite-rank operators, respectively. We will use three topologies on the spaces of operators: $\tau_c$ denotes the compact-open topology and SOT and WOT stand for the strong and weak operator topologies, respectively.

We recall the approximation properties we will use in the paper. A Banach space $X$ has the \emph{approximation property} (AP, for short) if $\id_X\in\overline{\mathcal F(X,X)}^{\tau_c}$, and it has the \emph{$\lambda$-bounded approximation property} ($\lambda$-BAP), for $\lambda\geq1$, if $\id_X\in\lambda\overline{B_{\mathcal F(X,X)}}^{\tau_c}$. Replacing $\mathcal F(X,X)$ by $\mathcal K(X,X)$ in these definitions yields the \emph{compact approximation property} (CAP) and the \emph{$\lambda$-bounded compact approximation property} ($\lambda$-BCAP), respectively. We shall also use the bounded approximation property for pairs of Banach
spaces, introduced in \cite{Bonder}. A pair $(X,Y)$ has the \emph{$\lambda$-bounded approximation property} ($\lambda$-BAP)  if $B_{\mathcal L(X,Y)}\subseteq\lambda\overline{B_{\mathcal F(X,Y)}}^{\tau_c}$. Replacing $\mathcal F(X,Y)$ by $\mathcal K(X,Y)$ defines the $\lambda$-BCAP for the pair; see \cite[Definition~2.1]{DJM}. For simplicity, we say that the pair $(X,Y)$ has the \emph{bounded approximation property} (resp., \emph{bounded compact approximation property}) if it has the $\lambda$-BAP (resp., $\lambda$-BCAP) for some $\lambda \geq 1$. If one of the spaces $X$ or $Y$ has the $\lambda$-BAP (respectively, the $\lambda$-BCAP), then the pair $(X,Y)$ has the $\lambda$-BAP (respectively, the $\lambda$-BCAP). The converse implications fail in general. Indeed, there are pairs with the BAP for which neither component has the BAP (see \cite[Example~4.2]{Bonder}). Likewise, a pair may have the BCAP even though neither component has the CAP (see \cite[Example~2.2]{DJM}).

\section{Relations among norm-attainment classes}\label{sec:preliminaries}
In this section we collect several general facts concerning the classes $\NA(X,Y)$, $\QNA(X,Y)$, $\NA_1(X,Y)$, and $\NA_2(X,Y)$ which will be used in the manuscript.

\subsection{General results}

Recall that an operator $T\in\mathcal L(X,Y)$ is said to be \emph{quasi-norm-attaining} if $\overline{T(B_X)} \cap \|T\|S_Y\neq \emptyset$. We denote the set of all such operators by $\QNA(X,Y)$. Observe that compact operators are always quasi-norm-attaining.

We first observe the following inclusions: 
\begin{equation}\label{eq:implications_NA}
    \NA (X,Y) \subseteq \QNA(X,Y)\subseteq  \NA_1 (X,Y) \subseteq \NA_2 (X, Y). 
\end{equation}
The first and the last inclusions are immediate, and the second one is easy to check (see  \cite[Proposition 3.3]{CCJM}). All the inclusions in \eqref{eq:implications_NA} can be checked to be strict already in the case $X=Y=c_0$ as follows. 
\begin{example}\label{ex:strict_inclusion}
    Fix $u = (1/2^n) \in S_{\ell_1}$ and a sequence $(\alpha_n)\subset(1/2,1)$ such that $\alpha_n\to1$. Let $(e_n)$ denote the canonical basis of $c_0$, and define $D\in \mathcal L(c_0,c_0)$ by $De_n:=\alpha_ne_n$ for every $n \in \mathbb{N}$.
    \begin{itemize}
        \itemsep0.25em
        \item If $T\colon c_0 \to c_0$ is defined as 
       $Tx = \langle u,x \rangle e_1$ for every $x \in c_0$, then $T \not\in \NA (c_0, c_0)$. However, $T \in \QNA(c_0,c_0)$ since it is a finite-rank operator (in particular, compact). 
       \item If $T\colon c_0\to c_0 \equiv\mathbb{R}\oplus_\infty c_0$ is defined as $Tx = (\langle u,x\rangle, Dx)$ for every $x \in c_0$, then $\|T\|=1$ and $T^* (1,0)=u$; hence $T \in \NA_1 (c_0,c_0)$. Notice that, for every $x \in c_0 \setminus \{0\}$,
       \[
       \|x\|>\|Tx\|\geq \|Dx\| \geq \frac{1}{2} \|x\|, 
       \]
       so $T$ is bounded below and does not belong to $\NA(c_0,c_0)$. Thus, $T \notin \QNA(c_0,c_0)$  by \cite[Lemma 2.1]{CCJM}.
    \item The operator $D$ does not belong to $\NA_1(c_0,c_0)$. On the other hand, $D^{**}$ attains its norm at $\zeta=(1,1,\ldots) \in S_{\ell_\infty}$. 
    \end{itemize}
    \end{example}

We next consider conditions under which some of these classes coincide. If the domain space $X$ is reflexive, then all three inclusions in \eqref{eq:implications_NA} are equalities. For reflexive range
spaces, we have the following related result.

\begin{proposition} \label{NA1-NA2-Y-reflexive} 
Let $X$ and $Y$ be Banach spaces. Then 
\begin{equation*}
    \QNA (X, Y) \cap \mathcal{W}(X,Y)  = {\NA_1 (X, Y) \cap \mathcal{W}(X,Y)  = \NA_2 (X,Y) \cap \mathcal{W}(X,Y).}
\end{equation*}
In particular, if $Y$ is reflexive, then $\QNA(X,Y)=\NA_1(X,Y) = \NA_2(X,Y)$.
\end{proposition}

\begin{proof} Let $T \in \NA_2(X,Y) \cap \mathcal{W} (X,Y)$ be given and let $x_0^{**} \in S_{X^{**}}$ such that $\|T^{**}(x_0^{**})\| = \|T\|$. Since $T$ is weakly compact, $T^{**}(x_0^{**}) \in Y$. Moreover, 
\[
T^{**} (x_0^{**}) \in \overline{T(B_X)}^w \cap \|T\|S_Y =  \overline{T(B_X)} \cap \|T\|S_Y
\]
because $T(B_X)$ is convex. Therefore, $T \in \QNA(X,Y)$.
\end{proof}

Proposition \ref{NA1-NA2-Y-reflexive} shows that reflexivity of the range space forces the classes $\QNA$, $\NA_1$, and $\NA_2$ to coincide. We now show that the equality between the first two classes already characterizes reflexivity, and that, for spaces not containing $\ell_1$, non-reflexivity also separates the latter two classes.

\begin{theorem}\label{thm:NA1_NA2_converse}
    Let $Y$ be a non-reflexive Banach space. Then 
    \begin{enumerate}[label=(\alph*)]
        \itemsep0.25em
        \item $\QNA(\ell_1,Y)\neq \NA_1 (\ell_1,Y)$; 
        \item If, in addition, $Y$ contains no isomorphic copy of $\ell_1$, then $\NA_1(\ell_1, Y) \not= \NA_2(\ell_1, Y)$.
    \end{enumerate}
\end{theorem}

\begin{proof}
(a) Let $y^* \in S_{Y^*}\setminus \NA(Y,\mathbb{R})$ and let $(y_n) \subseteq S_Y$ such that $y^* (y_n) \to 1$. Let $Z:=\overline{\text{span}}\{y_n \colon n\in\mathbb{N}\}$ and let $z^{**}$ be a weak$^*$-cluster point of $(y_n)$ in $B_{Z^{**}}$. Since $Z$ is separable, $B_{Z^*}$ is weak$^*$-metrizable; hence one may choose a weak$^*$-dense sequence $(h_n)$ in $\ker z^{**}\cap S_{Z^{*}}$. For each $n \in \mathbb{N}$, consider
\[
U_n:=\{u^{**} \in B_{Z^{**}}\colon |u^{**}(h_k)|< 1/n, \,k=1,\ldots , n\}
\]
which is a weak$^*$-open neighborhood of $z^{**}$. Choose an increasing sequence $(j_n)$ of natural numbers such that $y_{j_n} \in U_n$ for every $n\in\mathbb{N}$. For simplicity, put $z_n:= y_{j_n} \in S_Z$. Notice that 
\begin{equation}\label{eq:hk-zn}
    h_k (z_n) \to 0 \quad \text{as $n\to\infty$} 
\end{equation}
for each fixed $k \in \mathbb{N}$.

Next, let $(\beta_n)$ be a sequence in $(0,1)$ converging to $1$. Define $T\colon \ell_1\to Y$ by $Te_n:=\beta_n z_n$ for every $n \in \mathbb{N}$. First, note that $\|T\|=1$ and 
\[
1\geq \|T^{*} y^*\| \geq |y^* (Te_n)|=\beta_n |y^* (z_n)| \to 1.
\]
It follows that $T \in \NA_1 (\ell_1,Y)$. 

Suppose, for contradiction, that $T \in \QNA(\ell_1, Y)$. Let $(x_n) \subseteq B_{\ell_1}$ such that $(Tx_n) \to y$ for some $y \in S_Y$. In fact, the limit element $y$ belongs to $Z$. Writing $x_n = \sum_{j=1}^\infty \lambda_j^{(n)} e_j$ for each $n \in \mathbb{N}$, we have
\[
d_n := \sum_{j=1}^\infty (1-\beta_j) |\lambda_j^{(n)}| \leq 1 - \|Tx_n\| \to 0.
\]
Thus, for fixed $N \in \mathbb{N}$, 
\begin{equation}\label{eq:xn-support}
\sum_{j=1}^N |\lambda_j^{(n)}| \leq \frac{d_n}{\min_{1\leq i \leq N} (1-\beta_i)} \to 0 
\end{equation}
as $n\to\infty$. Therefore, for fixed $k\in\mathbb{N}$ and $N\in\mathbb{N}$, 
\begin{align*}
    |h_k (Tx_n) | &\leq \sum_{j=1}^N |\lambda_j^{(n)} | \beta_j |h_k(z_j) | + \sup_{j>N} |h_k (z_j)| \sum_{j>N} |\lambda_j^{(n)}| \beta_j \\
    &\leq  \sum_{j=1}^N |\lambda_j^{(n)} |  + \sup_{j>N} |h_k (z_j)|. 
\end{align*}
By letting $n \to\infty$, we then have from \eqref{eq:xn-support} that 
\[
|h_k (y)| \leq \sup_{j>N} |h_k (z_j)|.
\]
Finally, by letting $N\to\infty$, \eqref{eq:hk-zn} yields that $h_k(y)=0$. Since $k$ was arbitrary, we have that $h_k(y)=0$ for every $k \in \mathbb{N}$. Since $(h_k)$ is weak$^*$ dense in $\ker z^{**} \cap S_{Z^*}$, it follows that  $y \in (\ker z^{**})^\perp = \text{span}\{z^{**}\}$. This implies that $z^{**} \in Z \subseteq Y$. Since $z^{**}$ is a weak$^*$-cluster point of $(y_n)$, 
\[
z^{**} (y^* \restricted_{Z}) = 1;
\]
so this would imply that $y^*$ attains its norm at $z^{**} \in B_{Y}$, which is a contradiction. Consequently, we conclude that $T \not\in \QNA(\ell_1,Y)$.

(b) Choose a separable non-reflexive subspace $Z \subseteq Y$. Since $Z$ is non-reflexive, $Z^*$ is non-reflexive. By James' theorem, there is $z^{**} \in S_{Z^{**}}$ such that $|z^{**}(z^*)| < 1 $ for every $z^* \in B_{Z^*}$. Since $Z$ is separable and contains no copy of $\ell_1$, by the Odell-Rosenthal theorem (see \cite[Theorem~XIII.10]{Diestel-sequences}, for instance), $B_Z$ is weak$^*$ sequentially dense in $B_{Z^{**}}$. So, there exists a sequence $(x_n) \subseteq B_Z$ converging weak$^*$ to $z^{**}$. Since $\|z^{**}\| = 1$, we have that $\|x_n\| \rightarrow 1$. Write $z_n := (1-\frac{1}{n})x_n$ for every $n$. Then $\|z_n\| < 1$ for every $n$, $\|z_n\| \rightarrow 1$ and $z_n \rightarrow z^{**}$ weak$^*$. Define $T\colon  \ell_1 \rightarrow Y$ by $T e_n := z_n$ for every $n$. Then, $\|T\| = 1$ and $\|T^* y^*\| = \sup_n |(y^* \restricted_Z)(z_n)| < 1$ for every $y^* \in B_{Y^*}$ (otherwise, $z^{**}$ would attain its norm at $y^* \vert_Z)$. So, $T \not\in \NA_1(\ell_1, Y)$. Now, by taking a weak$^*$-cluster point $w$ of $(e_n)$ in $B_{\ell_1^{**}}$, we have $(T^{**} w)(y^*) = z^{**}(y^*\restricted_Z)$ for every $y^* \in B_{Y^*}$. Since $\|z^{**}\|=1$, we conclude that $\|T^{**}w\|=1$; hence $T \in \NA_2(\ell_1, Y)$.
\end{proof}

In fact, we get the following corollary. 

\begin{corollary}\label{cor:charac-Y-reflexive}
    Let $Y$ be a Banach space. Then the following are equivalent:
    \begin{enumerate}
        \itemsep0.25em
        \item $Y$ is reflexive.
        \item $\QNA(X,Y) = \NA_1(X,Y)$ for every Banach space $X$.
        \item $\QNA(\ell_1,Y) = \NA_1 (\ell_1,Y)$.
    \end{enumerate}
    If, moreover, $Y$ contains no isomorphic copy of $\ell_1$, then these conditions are also equivalent to 
\begin{enumerate}[label=(\arabic*), start=4] 
\itemsep0.25em
        \item $\NA_1(X,Y) = \NA_2(X,Y)$ for every Banach space $X$.
        \item $\NA_1(\ell_1,Y) = \NA_2 (\ell_1,Y)$.
    \end{enumerate}
\end{corollary}

Next, to transfer norm-attainment properties between successive adjoint levels, consider the transpose map
\[
\Xi_{X,Y}\colon \mathcal{L}(X,Y^*)\to \mathcal{L}(Y,X^*)
\]
defined by $\Xi_{X,Y}(T)=T^\top$ and 
$(T^\top y)(x):=(Tx)(y)$ for every $x\in X,\ y\in Y$.
The map $\Xi_{X,Y}$ is a surjective linear isometry, and
\[
\Xi_{Y,X}\circ\Xi_{X,Y}
=\operatorname{Id}_{\mathcal L(X,Y^*)}.
\]
In particular, $(T^\top)^\top=T$ for every
$T\in\mathcal L(X,Y^*)$.

\begin{proposition} \label{NA1-NA2-general} Let $X$ and $Y$ be Banach spaces, and let $T \in \mathcal{L}(X,Y^*)$. \begin{enumerate}
        \itemsep0.25em
        \item If $T^\top \in \NA(Y,X^*)$, then $T \in \NA_1 (X,Y^*)$. In other words,
\[
    (\Xi_{X,Y})^{-1}\bigl(\NA(Y,X^*)\bigr)
    \subseteq \NA_1(X,Y^*).
\]
\item If $T^\top \in \NA_1(Y,X^*)$, then $T \in \NA_2 (X,Y^*)$. In other words,
\[
(\Xi_{X,Y})^{-1}\bigl(\NA_1(Y,X^*)\bigr)
    \subseteq \NA_2(X,Y^*).
\]
\end{enumerate}
If $Y$ is reflexive, then both inclusions are equalities.
\end{proposition}

\begin{proof} 
(1) Note that $T^\top = T^* \circ J_Y$, where $J_Y \colon Y \hookrightarrow Y^{**}$ is the canonical embedding. Thus, if $\|T^\top y\|=\|T\|$ for some $y \in S_Y$, then 
\[
\| T^* ( J_Y (y)) \| = \|T^\top y\| = \|T\|, 
\]
so $T^*$ attains its norm at $J_Y (y) \in S_{Y^{**}}$. 

(2) Let $x_0^{**} \in S_{X^{**}}$ be such that $(T^\top)^*$ attains its norm at $x_0^{**}$. Since $(T^{\top})^* = J_Y^* \circ T^{**}$, we have 
\[
\|T^{**} x_0^{**}\| \geq \| (J_Y^* \circ T^{**} ) (x_0^{**})\| = \|T\|.
\]
This implies that $T^{**}$ attains its norm. 

If $Y$ is reflexive, then $J_Y$ is a surjective isometry, from which the desired equalities follow immediately.
\end{proof}

\begin{corollary} \label{holub-1-Y-reflexive} Let $X$ and $Y$ be Banach spaces. 
\begin{enumerate}
        \itemsep0.25em
        \item If $\NA(Y,X^*) = \mathcal{L}(Y,X^*)$, then $\NA_1 (X,Y^*) = \mathcal{L}(X,Y^*)$.
\item If $\NA_1(Y,X^*) = \mathcal{L}(Y,X^*)$, then $\NA_2 (X,Y^*) = \mathcal{L}(X,Y^*)$.
        \end{enumerate}
If $Y$ is reflexive, then the converses of both implications also hold.
\end{corollary}

\begin{corollary} \label{holub2} Let $X$ and $Y$ be Banach spaces. Suppose that $Y$ is reflexive and that $(Y,X^*)$ has the BCAP. If $\mathcal{K}(Y, X^*) \not= \mathcal{L}(Y, X^*)$ then $\NA_1(X, Y^*) \not= \mathcal{L}(X, Y^*)$.
\end{corollary}

\begin{proof} By \cite[Theorem B]{DJM} (our Theorem \ref{thm:DJM}), our assumption implies that $\NA( Y, X^*) \neq \mathcal{L}(Y,X^*)$. Since $Y$ is reflexive, by Corollary \ref{holub-1-Y-reflexive}, $\NA_1 (X, Y^*) \neq \mathcal{L}(X,Y^*)$.
\end{proof}

\begin{corollary}
    If $Y$ is an infinite-dimensional separable reflexive Banach space, then $\NA_2(\ell_1, Y) \neq  \mathcal{L}(\ell_1, Y)$.
\end{corollary}

\begin{proof}
    Since $Y^*$ is separable, it is linearly isometric to a subspace of $\ell_\infty$. This embedding is not compact because $Y^*$ is infinite-dimensional. Moreover, $\ell_\infty$ has the $1$-approximation property, so the pair $(Y^*,\ell_\infty)$ has the BCAP. Applying Corollary \ref{holub2} with $X =\ell_1$ and with $Y$ replaced by $Y^*$, we obtain $\NA_1(\ell_1, Y) \neq  \mathcal{L}(\ell_1, Y)$. Since $Y$ is reflexive, Proposition \ref{NA1-NA2-Y-reflexive} gives $\NA_1 (\ell_1,Y)=\NA_2 (\ell_1,Y)$, and the conclusion follows.
\end{proof}

It is easy to observe that if $X$ is reflexive, then $\mathcal{K}(X,Y) \subseteq \NA(X,Y)$ for every Banach space $Y$. This observation has a generalization to arbitrary domain spaces in the following result which is an easy application of Schauder's theorem. (A proof of a slightly stronger result can be found in \cite[Remark 3.23]{JungMartinRueda2023}.)

\begin{fact}\label{fact:compact_is_NA1}
    Let $X$ and $Y$ be Banach spaces. Then 
    \begin{equation*}
        \mathcal{K}(X,Y) \subseteq \NA_1 (X,Y).
    \end{equation*}
\end{fact}

We would also like to mention the following folklore result which is an immediate consequence of Rosenthal's $\ell_1$-theorem (see \cite[Theorem 11.2.1]{AlbiacKalton} for a reference to Rosenthal's theorem).

\begin{remark}\label{rem:Rosenthal} Let $X$ and $Y$ be Banach spaces. If $X$ does not contain an isomorphic copy of $\ell_1$ and $Y$ has the Schur property, then $\mathcal{L}(X,Y) = \mathcal{K}(X,Y)$. As a consequence, Fact~\ref{fact:compact_is_NA1} yields that 
    \[
    \mathcal{K}(X,Y)=\NA_1 (X,Y)=\NA_2 (X,Y)=\mathcal{L}(X,Y).
    \]
\end{remark}

We finish this subsection by presenting the following heredity lemma.

\begin{lemma}\label{lem:heredity}
    Let $X$, $Y$, and $Z$ be Banach spaces. Let $\Phi \colon Z \to Y$ be a linear isometry and $T \in \mathcal{L}(X,Z)$. Then 
\begin{enumerate}[label=(\alph*)]
    \itemsep0.25em
    \item $T \in \NA_1 (X,Z)$ if and only if $\Phi \circ T \in \NA_1 (X,Y)$.
    \item $T \in \NA_2 (X,Z)$ if and only if $\Phi \circ T \in \NA_2 (X,Y)$.
    \end{enumerate}
\end{lemma}

\begin{proof}
    (b) is immediate since $(\Phi \circ T)^{**} = \Phi^{**}\circ T^{**}$ and $\Phi^{**}$ is an isometry. To prove (a), suppose that $\|T^* z^*\|=\|T\|$ for some $z^* \in B_{Z^*}$. Then $z^* \circ \Phi^{-1} \colon \Phi(Z) \to \mathbb{R}$ is a bounded linear functional. By the Hahn-Banach theorem, $z^* \circ \Phi^{-1}$ admits a norm-preserving extension $\widetilde{z^* \circ \Phi^{-1}} \in Y^*$. Moreover, 
    \[
    (\Phi \circ T)^* ( \widetilde{z^* \circ \Phi^{-1}} ) = T^* z^*,
    \]
    so $\Phi \circ T \in \NA_1 (X,Y)$. Conversely, if $\|(\Phi \circ T)^* y^*\| = \|\Phi \circ T\|$ for some $y^* \in B_{Y^*}$, then $\| T^* (\Phi^*y^*)\| = \|T\|$; thus $T^*$ attains its norm at $\Phi^*y^* \in B_{Z^*}$.
\end{proof}

\subsection{When the target space is \texorpdfstring{$\ell_1$}{l1}}

Theorem \ref{thm:NA1_NA2_converse}(a) applies to $Y=\ell_1$ and yields 
\[
\QNA(\ell_1,\ell_1) \neq \NA_1 (\ell_1,\ell_1).
\]
However, part~(b) of Theorem \ref{thm:NA1_NA2_converse} does not cover the case $Y=\ell_1$. Moreover, the construction used in its proof cannot be adapted directly to this setting. Indeed, if a sequence in $\ell_1$ converges in the weak$^*$ topology of $\ell_1^{**}$, then its weak$^*$ limit necessarily belongs to $\ell_1$, by the Schur property. Hence that construction cannot produce a weak$^*$ limit outside $\ell_1$. It is therefore natural to ask whether $\NA_1(X,\ell_1)$ and $\NA_2(X,\ell_1)$ can nevertheless be different. The following answers this question affirmatively, even for $X=\ell_1$.

\begin{proposition}\label{prop:l1_l1}
    $\NA_1 (\ell_1, \ell_1) \neq \NA_2 (\ell_1,\ell_1)$. 
\end{proposition}

\begin{proof}
    Let $z_0 := (1,1,1,\ldots ) \in S_{\ell_\infty}$, $x_0 := (1/2^n ) \in S_{\ell_1}$, and define $T \colon \ell_1 \to \ell_1$ by $T = z_0 \otimes x_0 - \id_{\ell_1}$, that is,
    \[
    Tx = z_0 (x) x_0 - x \quad (x \in \ell_1).
    \]
    Note that $\|T\|\leq 2$ and letting $(e_n)$ be the canonical basis for $\ell_1$, 
    \[
    Te_n = x_0 - e_n  = \left(\frac{1}{2}, \ldots, \frac{1}{2^{n-1}}, \frac{1}{2^n} -1, \frac{1}{2^{n+1}}, \ldots \right). 
    \]
    Thus, $\|Te_n\| = 2-2^{1-n} \to 2$ as $n \to\infty$. It follows that $\|T\|=2$.
   We claim that $T \not\in \NA_1 (\ell_1,\ell_1)$. Assume to the contrary that $\|T^* \zeta\| = 2$ for some $\zeta \in S_{\ell_\infty}$. Then 
    \[
    2 = \sup_{n\in\mathbb{N}} |\zeta(x_0) z_0 (e_n) - \zeta (e_n)| = \sup_{n\in\mathbb{N}} |\zeta(x_0)  - \zeta (e_n)|.
    \]
    This implies that $|\zeta(x_0)|=1$; hence $\zeta = (1,1,1,\ldots)$ or $(-1,-1,-1,\ldots )$. In any case, $|\zeta(x_0)-\zeta(e_n)|=0$, which is a contradiction.
    It remains to prove that $T \in \NA_2 (\ell_1,\ell_1)$. To this end, let $\phi\in S_{\ell_\infty^*}$ be a Banach limit. Then, $T^{**} \phi = x_0 - \phi$. Since $\ell_{\infty}^* = \ell_1 \oplus_1 c_0^{\perp}$ (see Fact~\ref{fact:Hewitt-Yosida}), we have that $\|T^{**} \phi\| = \|x_0\| + \|\phi\| = 2$.
\end{proof}

As a direct consequence of the above result and Lemma \ref{lem:heredity}, we obtain the following.

\begin{corollary}
    If a Banach space $Y$ contains an isometric copy of $\ell_1$, then $\NA_1 (\ell_1, Y)\neq \NA_2 (\ell_1, Y)$.
\end{corollary}

\subsection{The Schur property}

The following proposition gives a characterization of the Schur property of $X^*$ in terms of weakly compact operators into $c_0$ and the classes $\NA_1$ and $\NA_2$.

\begin{proposition}\label{prop:Schur_equivalence}
    Let $X$ be a Banach space. Then the following are equivalent.
	\begin{enumerate}[label=(\alph*)]
    \itemsep0.25em
    \item $X^*$ has the Schur property.
    \item $\mathcal{W}(X,Z)=\mathcal{K}(X,Z)$ for every Banach space $Z$.
    \item $\mathcal{W}(X,c_0)=\mathcal{K}(X,c_0)$.
    \item $\mathcal{W}(X,c_0) \subseteq \NA_1 (X,c_0)$.
    \item $\mathcal{W}(X,c_0) \subseteq \NA_2 (X,c_0)$.
    \end{enumerate}
\end{proposition}

\begin{proof}
    (a) $\implies$ (b): If $T \in \mathcal{W}(X,Z)$, then $T^* \in \mathcal{W}(Z^*, X^*)$. Since $X^*$ has the Schur property, every weakly compact set in $X^*$ is norm compact. Thus, $T^*$ is compact and so is $T$.     The implications (b) $\implies$ (c) and (d) $\implies$ (e) are obvious. Note that (c) $\implies$ (d) follows from Fact~\ref{fact:compact_is_NA1}. It remains to prove that (e) $\implies$ (a). Assume to the contrary that $X^*$ does not have the Schur property. Then, there exists a sequence $(x_n^*)$ in  $S_{X^*}$ that weakly converges to $0$. Define the operator $T \colon X \to c_0$ by 
	\[
    (Tx)(n)=\frac{n}{n+1} x_n^*(x), \quad \forall x \in X, \ \forall n \in \N. 
    \]
	Note that $\|T\|=1$ and observe that for $z^{**} \in {X^{**}}$, 
    \[
    (T^{**} z^{**})(n) = \frac{n}{n+1} z^{**}(x_n^*) \to 0
    \]
    since $z^{**}(x_n^*)\to0$. It follows that $T^{**} z^{**} \in c_0$ for every $z^{**} \in X^{**}$; hence $T$ is weakly compact. On the other hand,  
	\[ \|T^{**} z^{**}\| = \sup_{n \in \N} |(T^{**}(z^{**}))(n)| = \sup_{n \in \N} \frac{n}{n+1} |z^{**}(x_n^*)| < \|z^{**}\|  \]
	for every $z^{**} \in X^{**} \setminus \{0\}$. Therefore, $T \notin \NA_2 (X, c_0)$.
\end{proof}

\begin{corollary}\label{cor:Schur} Let $X$ be a Banach space. If $\NA_2 (X, c_0)=\mathcal{L}(X, c_0)$, then $X^*$ has the Schur property. 
\end{corollary}

\section{Norm-attaining second adjoints on \texorpdfstring{$\ell_1$}{l1}-preduals}\label{sec:ell_1_preduals}

This section is devoted to the proof of Theorems \ref{thm:main_c0} and \ref{thm:main_W_alpha_W_beta} from the introduction. We begin in Subsection \ref{subsec:c_0_target} by establishing the key disjoint-support lemma (see Lemma \ref{lemma:disjoint_support_lemma}) and using it to prove that the equalities $\mathcal{L}(c_0,c_0)=\NA_2 (c_0, c_0)$ and $\mathcal{L}(c,c_0)=\NA_2 (c,c_0)$ hold true. The next two subsections carry out a systematic study of the equality $\mathcal{L}(W_\alpha,W_\beta)=\NA_2 (W_\alpha,W_\beta)$ for $\ell_1$-preduals within the class $\{W_\alpha\}$. In Subsection \ref{subsec:c_target}, we consider the case of the target space $c$ (see Theorem \ref{thm:target_c}), and in Subsection \ref{subsec:W_beta_target}, we handle the general case $W_\beta$ (see Theorem \ref{thm:target_W_beta}). Together, these results constitute the complete picture described in Theorem \ref{thm:main_W_alpha_W_beta}. Finally, Subsection \ref{subsec:l1} is devoted to proving the $\ell_1$ part of Theorem \ref{thm:main_c0}.

\subsection{The target space \texorpdfstring{$c_0$}{c0}}\label{subsec:c_0_target}
The positive results in this section are all based on a common principle: after passing to a suitable subsequence, one may perturb a sequence in $\ell_1$ so that the supports become disjoint while preserving its norm asymptotically. This disjointness allows us to construct an element of $S_{\ell_\infty}$ that simultaneously attains the norm on the perturbed vectors, which in turn yields a norm-attaining point for the second adjoint.

\begin{lemma} \label{lemma:disjoint_support_lemma} Let $(a_n)$ be a bounded sequence in $\ell_1$ such that $\lim_{n\to\infty} a_n(k) = 0$ for every $k \in \N$. Then, there is a sequence $(b_n)$ of finitely supported vectors in $\ell_1$ such that 
	\begin{enumerate}[label=(\alph*)]
    \itemsep0.25em
		\item $\supp(b_n) \cap \supp(b_m) = \varnothing$ for every $n \not= m$ and 
		\item there exists an increasing sequence $(n_j)$ of natural numbers such that $\|a_{n_j} - b_{j}\| \to 0$ as $j \to \infty$.
	\end{enumerate}
    Moreover, one may arrange that $\|b_j\|\leq \|a_{n_j}\|$ for every $j \in \mathbb{N}$.
\end{lemma}

The result follows from the classical gliding hump argument of the Bessaga--Pełczyński Selection Principle (see \cite[Proposition 1.3.10]{AlbiacKalton}, for instance) applied to $\ell_1$. We include a direct proof for the sake of completeness.

\begin{proof} 
Choose $n_1 \in \mathbb{N}$. Since $a_{n_1} \in \ell_1$, there exists a finite set $F_1 \subseteq \mathbb{N}$ such that 
\begin{equation}\label{eq:a_n_1}
\sum_{k \notin F_1} |a_{n_1} (k) | < \frac{1}{2}.
\end{equation}
Put $b_1 := a_{n_1} \chi_{F_1}$, i.e., $b_1 (k) = a_{n_1} (k)$ for $k \in F_1$ and $b_1 (k) =0$ otherwise. Note that \eqref{eq:a_n_1} implies that $\|a_{n_1}-b_1 \| < 1/2$.

 Suppose that, for some $j \geq 1$, an increasing sequence
$n_1, \ldots, n_j$ and pairwise disjoint finite sets $F_1, \ldots, F_{j} \subseteq \mathbb{N}$ have already been chosen such that $b_1 = a_{n_1} \chi_{F_1}, \ldots, b_j = a_{n_j} \chi_{F_j}$ satisfy that 
 \[
 \|a_{n_i} - b_i\| < \frac{1}{2i} \quad \text{ for } i=1,\ldots, j.
 \]
  Let 
 \[
 H_{j}:= F_1 \cup \ldots \cup F_{j}.
 \]
 Since $H_{j}$ is finite and $a_n(k) \rightarrow 0$ for each fixed $k$, we have that 
\begin{equation*}
    \sum_{k \in H_{j}} |a_n(k)| \rightarrow 0
\end{equation*}
as $n \rightarrow \infty$. Therefore, we may choose $n_{j+1} > n_{j}$ large enough so that 
\begin{equation}\label{eq:H_j}
    \sum_{k \in H_{j}} |a_{n_{j+1}}(k)| < \frac{1}{4(j+1)}.
\end{equation}
Since $a_{n_{j+1}} \in \ell_1$, there exists a finite set $G_{j+1} \subseteq \mathbb{N}$ such that 
\begin{equation}\label{eq:G_j+1}
    \sum_{k \not\in G_{j+1}} |a_{n_{j+1}}(k)| < \frac{1}{4(j+1)}.
\end{equation}
Define $F_{j+1}:= G_{j+1} \setminus H_{j}$ and $b_{j+1} := a_{n_{j+1}} \chi_{F_{j+1}}$. Then, $F_{j+1}$ is finite (possibly empty) and disjoint from $F_1, \ldots, F_{j}$. Note that 
\begin{equation*}
    \|b_{j+1}\|_1 = \sum_{k \in F_{j+1}} |a_{n_{j+1}}(k)| \leq \|a_{n_{j+1}}\|_1.
\end{equation*}
Since $F_{j+1}^c = (G_{j+1} \setminus H_{j})^c = G_{j+1}^c \cup H_j$, 
we have from \eqref{eq:H_j} and \eqref{eq:G_j+1} that 
\begin{align*}
    \|a_{n_{j+1}} - b_{j+1}\|_1 = \sum_{k \not\in F_{j+1}} |a_{n_{j+1}}(k)| &\leq \sum_{k \notin G_{j+1}} |a_{n_{j+1}}(k)| + \sum_{k \in H_j}  |a_{n_{j+1}}(k)|     < \frac{1}{2(j+1)}.
\end{align*}
Proceeding inductively, we obtain the required sequences.
\end{proof}

We obtain the following useful consequence, which was suggested to us by Vladimir Kadets.

\begin{corollary}\label{cor:finite-head-norming}
    Let $(a_n)$ be a bounded sequence in $\ell_1$ such that
$\|a_n\|\to 1$. Suppose that there exists a finite set
$F\subset\mathbb N$ such that
\[
a_n(k)\to 0 \quad \text{ whenever } k \not\in F.
\]
Then, there exist a subsequence $(a_{n_j})$ and
$\zeta\in S_{\ell_\infty}$ such that
\[
\zeta(a_{n_j})\to 1.
\]
\end{corollary}

\begin{proof}
    Passing to a subsequence, we may suppose that
$(a_n \restricted_{F})$ converges to some $h \in \ell_1$ in norm. Apply Lemma \ref{lemma:disjoint_support_lemma} to
$(a_n \restricted_{F^c})$ and obtain an increasing sequence $(n_j)$ of natural numbers and pairwise disjointly supported vectors $(b_j)$, supported outside $F$, such that
\[
\| a_{n_j} \restricted_{F^c} -b_j\|\to 0.
\]
Since $\|a_{n_j}\|
 =\| a_{n_j} \restricted_{F} \|+\| a_{n_j} \restricted_{F^c} \|$,
we have
\[
\|b_j\|\to 1-\|h\|.
\]
Define $\zeta\in S_{\ell_\infty}$ such that $\zeta(h)=\|h\|$ and $\zeta(b_j)=\|b_j\|$ for every $j \in \mathbb{N}$. Then
\[
\zeta(a_{n_j})\to \|h\|+(1-\|h\|)=1. \qedhere
\]
\end{proof}

We are now ready to give the first main result of the section.

\begin{theorem} \label{theorem:c0} The following statements hold. \begin{enumerate}[label=(\arabic*)] 
\itemsep0.25em
\item $\mathcal{L}(c_0, c_0)= \NA_2(c_0, c_0)$. 
\item $\mathcal{L}(c, c_0)= \NA_2(c, c_0)$. 
\end{enumerate} 
\end{theorem}

\begin{proof}
(1) Let $T \in \mathcal{L}(c_0, c_0)$ and assume without loss of generality that $\|T\| = 1$. Set $a_n:= T^*(e_n)$ for every $n \in \N$, where $(e_n)$ is the canonical basis for $\ell_1$. Note that $a_n(k) \to 0$ as $n \rightarrow \infty$ for every $k \in \N$.

If $T^*$ attains its norm, then we are done. Suppose that $T^*$ does not attain its norm. In this case, we have that $1= \limsup_n \|T^*(e_n)\| = \limsup_n \|a_n\|$. Passing to a subsequence and relabeling if necessary, we may assume that $\|a_n\| \to 1$. 
    By Corollary \ref{cor:finite-head-norming} (with $F=\varnothing$), we obtain an increasing function $\tau\colon \N \rightarrow \N$ and $\zeta \in S_{\ell_\infty}$ such that 
    \[
    \zeta(a_{\tau(n)}) \to 1.
    \]
    Therefore, 
    	\begin{align*}
		\|T^{**} \zeta\| \geq |(T^{**}\zeta)(e_{\tau(n)})| &= |\zeta(T^*(e_{\tau(n)}))| = |\zeta(a_{\tau(n)})| \to 1,
	\end{align*}
which proves that $T^{**}$ attains its norm.

    (2) Let $T \in \mathcal{L}(c,c_0)$ and assume that $\|T\|=1$. Let $\Phi \colon \ell_1 \to c^*$ be the canonical isometry in \eqref{eq:c_dual}. Letting $(e_n)$ and $(e_n^*)$ be the standard coordinate bases for $c_0$ and $\ell_1$ respectively, write $\Phi^{-1} (T^*(e_n^*)) = g_n$ for some $g_n \in \ell_1$ for every $n \in \mathbb{N}$. Note that for $k\geq 2$
\[
g_n (k) = \Phi (g_n) (e_{k-1}) = (T^* (e_n^*)) (e_{k-1}) = e_n^* (T (e_{k-1})) \to 0
\]
as $n\to\infty$.

Note that $\sup_{n} \|g_n\| = \|T\|=1$. If this supremum is attained at some $n$, then $T^*$ attains its norm at $e_n^*$, so $T \in \NA_1 (c, c_0)$. Otherwise, passing to a subsequence, we may assume that $\|g_n\|\to 1$ as $n\to \infty$.
Applying Corollary \ref{cor:finite-head-norming} with $F=\{1\}$, we obtain an increasing function $\tau\colon \mathbb{N}\to\mathbb{N}$ and $\zeta \in S_{\ell_\infty}$ such that 
\[
\zeta(g_{\tau(n)} ) \to 1.
\]
Set $x^{**} := (\Phi^* )^{-1} (\zeta) \in S_{c^{**}}$. Then
\begin{align*}
    \|T^{**}  x^{**}\| \geq \sup_{n \in \mathbb{N}} (T^{**} x^{**}) (e_{\tau(n)}^*) 
    = \sup_{n \in \mathbb{N}} \zeta ( \Phi^{-1} (T^* e_{\tau(n)}^* ) ) \geq \lim_{n\to\infty} \zeta (g_{\tau(n)}) = 1.
\end{align*}
Thus, $T^{**}$ attains its norm at $x^{**}$. 
\end{proof}

\begin{remark}\label{rem:c0_sum_finite_dim}
The proof of Theorem \ref{theorem:c0} extends to certain vector-valued $c_0$-sums. For instance, if $(E_n)_{n \in \mathbb{N}}$ is a sequence of finite-dimensional Banach spaces and
\[
X=
\left(\bigoplus_{n \in \mathbb{N}} E_n\right)_{c_0}
\]
then $\NA_2(X,X)=\mathcal L(X,X)$.
Indeed, the proof follows essentially the same disjoint-support argument used for the pair $(c_0, c_0)$ in Theorem \ref{theorem:c0}. In particular, this shows that the equality $\NA_2(X,X)=\mathcal L(X,X)$ is not restricted to preduals of $\ell_1$: for instance, $X =(\oplus_{n\in\mathbb{N}} \ell_1^n)_{c_0}$ satisfies this equality, while $X$ is not isomorphic to any predual of $\ell_1$. Indeed, $X^*$ contains $\ell_\infty^n$'s uniformly \cite[Corollary 7.7]{TJ} and therefore fails cotype $2$, whereas $\ell_1$ has cotype $2$.  
\end{remark} 

\begin{remark}\label{rem-c0gamma}
The proof of Theorem \ref{theorem:c0} also extends to spaces $c_0(\Gamma)$. Let $\Gamma$ be an arbitrary (infinite) set. Then  $$\mathcal{L}(c_0(\Gamma), c_0(\Gamma))= \NA_2(c_0(\Gamma), c_0(\Gamma)).$$ 
\end{remark}

The proof of Theorem \ref{theorem:c0} suggests that second-adjoint norm attainment could be related to the $\ell_1$-predual structure rather than to the specific geometry of $c_0$. We now investigate to what extent this phenomenon persists within the family of $\ell_1$-preduals $W_\alpha$ considered in \cite{CMP15a}.

\begin{definition}\label{def:Walpha}    
Given $\alpha \in S_{\ell_1}$, the space $W_{\alpha}$ is defined as the hyperplane of $c$ given by $W_\alpha=\ker(\Phi(\alpha)) \subseteq c$, that is,
\begin{equation*}
W_{\alpha} = \{x \in c\colon \Phi(\alpha)(x) = 0\} = \left\{ x \in c\colon \alpha_1 \lim_n x_n + \sum_{k=1}^{\infty} \alpha_{k+1} x_k = 0 \right\},
\end{equation*}
where $\Phi\colon \ell_1 \rightarrow c^*$ is the canonical isometry given by 
	\begin{equation}\label{eq:c_dual}
	\Phi(g)(x) = g_1 \lim_n x_n + \sum_{k=1}^{\infty} g_{k+1} x_k
	\end{equation}
	for every $g=(g_n) \in \ell_1$ and $x=(x_n) \in c$. 
\end{definition}    
The family $\{W_\alpha\colon \alpha\in S_{\ell_1}\}$ provides a natural testing ground for questions concerning norm attainment, as it contains representatives of the isometric classes of $c_0$, $c$, and many other $\ell_1$-preduals.

For the reader's convenience, we recall the classification of the spaces
$W_\alpha$ for which $W_\alpha^*$ is isometric to $\ell_1$.

\begin{remark}[\text{\cite[Proposition 4.1, Remark 4.2]{CMP15a}}] \label{remark-when-W-is-a-predual-of-l1} 
$W_{\alpha}^* \equiv \ell_1$ if and only if $|\alpha_j| \geq 1/2$ for some $j \in \N$. The case $W_{\alpha}^* \equiv \ell_1$ can be divided into three cases.
  \begin{enumerate}[label=(\roman*)]
\itemsep0.25em
	\item $W_{\alpha} \equiv c_0$ if and only if $|\alpha_1| = 1$.	
	\item $W_{\alpha} \equiv c$ if and only if there is $k\geq2$ such that $|\alpha_k| \geq 1/2$.	
	\item $W_{\alpha}$ is isometric to neither $c$ nor $c_0$, that is, when $1/2 \leq |\alpha_1| < 1$ and $|\alpha_j| < 1/2$ for every $j \geq 2$.
\end{enumerate}
\end{remark} 

We will also need the following remarks (see \cite[Remark 5.1]{CMP}).

\begin{remark} \label{remark:permutation} Let $\alpha, \beta \in B_{\ell_1}$. Then the spaces $W_{\alpha}$ and $W_{\beta}$ are isometric if and only if there exist a finite sequence of signs $\{\eps_n\}_{n=1}^{j_0}$ with $\eps_n = \pm 1$, and a permutation $\pi\colon \N \rightarrow \N$ such that $\alpha(n) = \eps_n \beta( \pi(n))$ for $1 \leq n \leq j_0$ and $\alpha(n) = \beta(\pi(n))$ for $n > j_0$.
\end{remark}

\begin{remark}\label{remark:onto} 
Let $\alpha \in S_{\ell_1}$ be given. Suppose that $W_\alpha^*$ is isometric to $\ell_1$, and that $1/2\leq |\alpha_1| < 1$ and $|\alpha_j|<1/2$ for every $j\geq 2$. In this case, the canonical map 
\begin{equation}\label{eq:W_alpha_dual}
	(\phi(y))(x) = \sum_{j=1}^{\infty} x_j y_j, 
	\end{equation}
    where $y = (y_n) \in \ell_1$ and $x = (x_n) \in W_\alpha$, is a surjective linear isometry; see the proof of \cite[Theorem 4.3]{CMP15a}.
\end{remark} 

The following lemma will be useful in the proof of Theorem \ref{full-characterization_target_c0}.

\begin{lemma}\label{lem:weak-star-limit-infinite-support}
    Let $X$ be a Banach space and let $\varphi \colon \ell_1 \to X^*$ be a surjective isometry. Let $(e_n)$ be the canonical basis of $\ell_1$. Suppose that, under this identification, 
    \[
    (e_n) \xrightarrow{\sigma(\ell_1, X)} \eta 
    \]
    for some infinitely supported $\eta \in \ell_1$ satisfying $|\eta_n| <1$ for every $n \in \mathbb{N}$. Then $\NA_2 (X, c_0) \neq \mathcal L (X,c_0)$. 
\end{lemma}

\begin{proof}
    Since $\eta$ is infinitely supported, one of the sets $\{j \in \mathbb{N}\colon \eta_j >0\}$ and $\{ j \in \mathbb{N} \colon \eta_j <0\}$ is infinite. Thus, there exist $\theta \in \{-1,1\}$ and an increasing function $\tau\colon \mathbb{N}\to\mathbb{N}$ such that $\theta \eta_{\tau(n)}>0$ for every $n \in \mathbb{N}$. That is, $\theta = \sign(\eta_{\tau(n)})$ for every $n \in\mathbb{N}$.

    Consider 
    \[
    p_\theta := \frac{3-\theta}{2} \quad \text{ and } \quad  q_\theta = \frac{3+\theta}{2}.
    \]
    Define 
    \[
    a_n := \theta \eta + p_\theta e_{\tau(n)} - q_\theta e_{\tau(n+1)} \in \ell_1.
    \]
    Note that $p_\theta-q_\theta=-\theta$ and $(e_{\tau(n)}) \xrightarrow{\sigma(\ell_1,X)} \eta$. It follows that 
    \[
    a_n \xrightarrow{\sigma(\ell_1,X)} \theta\eta + (p_\theta-q_\theta) \eta =0.
    \]
    Therefore, the formula 
    \[
    Tx := (\varphi(a_n)(x))_{n\in\mathbb N}
    \]
    defines a well-defined bounded linear operator from $X$ to $c_0$.

    Note that $T^*e_n^*= \varphi (a_n)$ for every $n \in \mathbb{N}$. Moreover, 
    \begin{align*}
        \|a_n\| &= (\|\eta\| - |\eta_{\tau(n)}| - |\eta_{\tau(n+1)}|) + (|\eta_{\tau(n)}| + p_{\theta}) + (q_\theta -|\eta_{\tau(n+1)} |)\\
        &=\|\eta\| +3 - 2|\eta_{\tau(n+1)}| < \|\eta\| + 3 
    \end{align*}
    for every $n \in \mathbb{N}$. This proves that $\|T\| = \sup_{n\in\mathbb{N}} \|T^*e_n^* \| = \|\eta\| +3$.

    Assume to the contrary that $T \in \NA_2 (X, c_0)$. Identifying $X^{**}$ with $\ell_\infty$ through $\varphi^{*}$, we see that there exists $\zeta \in B_{\ell_\infty}$ such that $\sup_{n\in\mathbb{N}} |\zeta(a_n) | = \|\eta\| +3$. Since 
    \[
    |\zeta(a_n)| \leq |\zeta(\eta)| + p_\theta + q_\theta \leq \|\eta\| +3,
    \]
    it follows that $|\zeta(\eta)|=\|\eta\|$. This implies that there exists $\epsilon \in \{-1,1\}$ such that $\zeta_j = \epsilon \sign (\eta _j) $ for every $j \in \mathbb{N}$ with $\eta_j \neq 0$. In particular, $\zeta(\eta)= \epsilon \|\eta\|$, and 
    \[
    \zeta_{\tau(n)} = \zeta_{\tau(n+1)} = \epsilon \sign (\eta_{\tau(n)}) = \epsilon \theta \quad (n\in\mathbb{N}).
    \]
    Therefore, for every $n \in \mathbb{N}$, 
    \begin{align*}
        \zeta(a_n) &= \theta \zeta(\eta) + p_{\theta} \zeta(e_{\tau(n)}) - q_\theta \zeta(e_{\tau(n+1)} )\\
        &=\theta \epsilon \|\eta\| + \epsilon \theta (p_{\theta}-q_\theta) = \theta \epsilon (\|\eta\| -\theta ).
    \end{align*}
    This implies that $\sup_{n\in\mathbb{N}}|\zeta(a_n)| \leq \|\eta\|+1$, which is a contradiction.
\end{proof}

We can now state and prove the first main result on $W_\alpha$-spaces.

\begin{theorem} \label{full-characterization_target_c0} 
Let $\alpha \in S_{\ell_1}$ be given. Suppose that $W_\alpha^* \equiv \ell_1$ and that $W_{\alpha}$ is isometric neither to $c$ nor $c_0$. 
	\begin{enumerate}[label=(\arabic*)]
     \itemsep0.25em
		 \item If $\alpha$ is finitely supported, then $\NA_2(W_{\alpha}, c_0) = \mathcal{L}(W_{\alpha}, c_0)$.  
		\item If $\alpha$ is infinitely supported, then $\NA_2(W_{\alpha}, c_0) \not= \mathcal{L}(W_{\alpha}, c_0)$. 
	\end{enumerate}
\end{theorem}

\begin{proof} 
(1) Suppose that $\alpha = (\alpha_1, \ldots, \alpha_N, 0, 0, \ldots) \in S_{\ell_1}$ for some $N \in \mathbb{N}$. Recall from Remark \ref{remark:onto} that the duality between $\ell_1$ and $W_{\alpha}$ is given by the isometry $\phi\colon \ell_1 \rightarrow W_{\alpha}^*$ defined by 
	\begin{equation*}
	\phi(y)(x) = \sum_{n=1}^{\infty} y_n x_n 
	\end{equation*}
	for every $y = (y_n) \in \ell_1$ and $x = (x_n) \in W_{\alpha}$. Let $T\colon W_{\alpha} \rightarrow c_0$ be given with $\|T\|=1$. Then $\sup_n \|T^* (e_n^*) \| = 1$. If this supremum is attained for some $n \in\mathbb{N}$, then $T \in \NA_1 (W_\alpha, c_0)$. Suppose that the supremum is not attained. By passing to a subsequence, we may assume that $\|T^*e_n^*\| \rightarrow 1$ as $n \rightarrow \infty$. 
    
    For simplicity, put $a_n := \phi^{-1}(T^* e_n^*)$ for every $n \in \mathbb{N}$. Note that $e_j \in W_\alpha$ for every $j \geq N$. Thus, for $j \geq N$, 
    \[
    a_n (j) = \phi(a_n) (e_j) =( T^* e_n^*)( e_j) = e_n^* (Te_j) \to 0
    \]
    as $n\to\infty$. 
    
    Applying Corollary \ref{cor:finite-head-norming} with $F=\{1,\ldots,N-1\}$, we obtain a subsequence $(a_{n_j})$ and $\zeta \in S_{\ell_\infty}$ such that 
    \[
    \zeta(a_{n_j}) \to 1.
    \]
    Set $x^{**} = (\phi^{-1})^*\zeta\in S_{W_\alpha^{**}}$. Then 
\[
\|T^{**} x^{**}\| \geq \lim_{j\to\infty} |(T^{**} x^{**} ) (e_{n_j}^*)| = \lim_{j\to\infty} |\zeta ( \phi^{-1} (T^* e_{n_j}^*) ) | = \lim_{j\to\infty} |\zeta(a_{n_j} )| =1.
\]
Therefore, $T^{**}$ attains its norm. 
    
(2) Suppose that $\alpha$ is infinitely supported. Without loss of generality, assume that $\alpha_1 > 0$ (see \cite[Remark 5.1]{CMP} or Remark \ref{remark:permutation}). Define 
\[
\eta_\alpha:= \left( -\frac{\alpha_2}{\alpha_1}, - \frac{\alpha_3}{\alpha_1}, - \frac{\alpha_4}{\alpha_1}, \ldots \right) \in \ell_1.
\]
Then for every $x=(x_j) \in W_\alpha$, we have 
\[
\phi(e_n)(x) = x_n \to \lim_{j\to\infty} x_j = -\frac{1}{\alpha_1} \sum_{j=1}^\infty \alpha_{j+1} x_j = \phi(\eta_\alpha)(x).
\]
This shows that $(e_n) \xrightarrow{\sigma(\ell_1,W_\alpha)} \eta_\alpha$. Since $\eta_\alpha$ is infinitely supported, and $|\eta_\alpha(j)|<1$ for every $j \in \mathbb{N}$, we apply Lemma \ref{lem:weak-star-limit-infinite-support} to conclude that 
\[
\NA_2 (W_\alpha,c_0) \neq \mathcal{L}(W_\alpha, c_0). \qedhere
\]
\end{proof}

As a consequence, we may show that the converse of Corollary \ref{cor:Schur} fails. 

\begin{example} \label{exa:the_converse_of_Schur_result}
There exists a Banach space $X$ such that $\NA_2(X,c_0)\neq \mathcal L(X,c_0)$ and $X^*$ has the Schur property. Indeed, choose $\alpha= (2^{-n}) \in S_{\ell_1}$ and take $X = W_\alpha$. By Theorem \ref{full-characterization_target_c0}, $\NA_2 (X,c_0) \neq \mathcal{L}(X,c_0)$ while $X^*$ is isometric to $\ell_1$.
\end{example}

\subsection{The target space \texorpdfstring{$c$}{c}}\label{subsec:c_target}

The disjoint-support argument used in the proof of Theorem \ref{theorem:c0}
relies heavily on the particular dual structure of $c_0$. It is therefore natural to ask whether the same phenomenon persists when the target space is replaced by $c$, another classical $\ell_1$-predual. In this subsection, we show that the answer is negative.

The following lemmas will be useful in the sequel.

\begin{lemma}\label{lem:isometric-conjugacy}
Let $X$ and $Y$ be Banach spaces, and let
\[
U\colon \ell_1\to X^*,
\qquad
V\colon \ell_1\to Y^*
\]
be surjective linear isometries. Given $T\in\mathcal L(X,Y)$, put $S:=U^{-1}T^*V\in\mathcal L(\ell_1,\ell_1)$.
    \begin{equation*}
    \begin{tikzcd}
	{Y^* } && {X^*} \\
	{\ell_1} && {\ell_1}
	\arrow["{T^*}", from=1-1, to=1-3]
	\arrow["{U^{-1}}", from=1-3, to=2-3]
	\arrow["V", from=2-1, to=1-1]
	\arrow["S"', from=2-1, to=2-3]
\end{tikzcd}    
    \end{equation*}
Then $\|S\|=\|T\|$, and $T \in \NA_2 (X,Y)$ if and only if $S^* \in \NA (\ell_\infty,\ell_\infty)$. 
\end{lemma}

\begin{proof}
We have $S^*=V^*T^{**}(U^{-1})^*$. Since $V^*$ and $(U^{-1})^*$ are surjective linear isometries, norm
attainment of $S^*$ is equivalent to norm attainment of $T^{**}$.
\end{proof}

\begin{lemma}\label{lem:moving-bump}
Let
\[
u:=\left(\frac1{2^n}\right)_{n=1}^\infty\in S_{\ell_1},
\qquad
b_n:=\frac12(e_{n+1}-e_{n+2})\in S_{\ell_1},
\]
where $(e_n)$ denotes the standard coordinate basis for $\ell_1$. 
Let $0<\kappa<1$, and let $S\in\mathcal L(\ell_1,\ell_1)$.
Suppose that there exist a set $A\subset\mathbb N$, a mapping
$\nu\colon A\to\mathbb N$, and scalars $(\lambda_j)_{j\in A}$ such that

\begin{enumerate}
\itemsep0.25em
\item  $Se_j=\kappa u+(1-\kappa)\lambda_j b_{\nu(j)}$ for every $j \in A$, and $|\lambda_j|\leq 1$ for every $j \in A$;

\item there exist $(j_n)\subset A$ and a sequence $(s_n)$ tending to
infinity such that
\[
Se_{j_n}=\kappa u+(1-\kappa)b_{s_n};
\]

\item $\sup_{j\notin A}\|Se_j\|<1$.
\end{enumerate}

Then $\|S\|=1$ and $S^*$ does not attain its norm.
\end{lemma}

\begin{proof}
The assumptions give $\|Se_j\|\le1$ for every $j \in \mathbb{N}$. Moreover, observe that 
\[
\|Se_{j_n} \| = \|\kappa u+(1-\kappa)b_{s_n}\| \to \kappa\|u\|+(1-\kappa)\|b_{s_n}\|=1.
\]
Hence $\|S\|=1$.

Suppose that $\|S^*\zeta\|=1$ for some $\zeta=(\zeta_n)\in S_{\ell_\infty}$. By (3), we have that 
\[
\begin{aligned}
1
=\sup_{j\in A}
 \left|
 \kappa \zeta(u)+(1-\kappa)\lambda_j \zeta(b_{\nu(j)})
 \right| &\leq
\kappa|\zeta(u)|+(1-\kappa)\sup_n|\zeta(b_n)|
\le1.
\end{aligned}
\]
It follows that $|\zeta(u)|=1$. Therefore, either $\zeta_n=1$ for every $n \in \mathbb{N}$, or $\zeta_n=-1$ for every $n \in \mathbb{N}$.
Consequently, $\zeta(b_n)=0$ for every $n \in \mathbb{N}$. The preceding supremum is therefore equal to $\kappa<1$, which is a
contradiction.
\end{proof}

Here is the second main result concerning $W_\alpha$-spaces.

\begin{theorem}\label{thm:target_c}
    Let $\alpha \in S_{\ell_1}$ be such that $W_\alpha^* \equiv \ell_1$. Then $\NA_2 (W_\alpha, c) \neq \mathcal{L}(W_\alpha, c)$.
\end{theorem}

\begin{proof}
\underline{Case 1}: Suppose that $W_\alpha$ is isometric neither to $c$ nor $c_0$. 
Set \[
u=\left(\frac{1}{2^n} \right), \, b_n = \frac{1}{2} (e_{n+1}-e_{n+2}) \, \text{ and } a_n = \frac{1}{2} u + \frac{1}{2} b_n \quad (n \in \mathbb{N}),
\]
where $(e_n)$ denotes the standard coordinate basis for $\ell_1$. 
    Define $T \in \mathcal{L}(W_\alpha, c)$ by 
    \begin{equation}\label{eq:definition_of_T_target_c}
    T(x) = ((Tx)(n))_{n\in\mathbb{N}}, \text{ where } \, (Tx)(n) := \phi(a_n)(x) \text { for } x \in W_\alpha.
    \end{equation}
    Note that for $x=(x_n) \in W_\alpha$, 
    \[
    (Tx)(n)  = \phi \left(\frac{1}{2}u \right)(x) + \frac{1}{2} \phi(b_n)(x) = \phi\left(\frac{1}{2}u \right)(x) + \frac{x_{n+1}}{4} - \frac{x_{n+2}}{4}.
    \]
    It follows that $(Tx)(n)\to \phi(\frac{1}{2}u)(x)$ as $n\to\infty$, so $T(x) \in c$ for every $x \in W_\alpha$. Since $\sup_{n} \|a_n\| \leq 1$, $T$ is bounded. 
    
     Consider the linear isometries $\phi\colon \ell_1 \to W_\alpha^*$ and $\Phi \colon \ell_1 \to c^*$ from \eqref{eq:W_alpha_dual} and \eqref{eq:c_dual}. Let $S:=  \phi^{-1} \circ T^* \circ \Phi$. Observe that for $x \in W_\alpha$
    \[
    [T^* (\Phi (e_k)) ] (x)= [\Phi (e_k)  ] (T(x)) = 
    \begin{cases}
        \lim_{n} (Tx)(n) \quad &\text{if } k =1; \\
        T(x)(k-1) \quad &\text{if } k \geq 2. 
    \end{cases}
    \]
    It follows that 
    \begin{equation*}
    T^* (\Phi (e_1)) = \frac{1}{2} \phi (u) \quad \text{ and } \quad T^* (\Phi (e_{k+1})) = \phi(a_k), \,\, \forall k\in\mathbb{N}.    
    \end{equation*}
    Therefore, 
\begin{equation*}
 S(e_1) = \frac{1}{2} u \quad \text{ and } \quad S(e_{k+1}) = a_k, \,\, \forall k\in\mathbb{N}.       
\end{equation*}
Applying Lemma \ref{lem:moving-bump} with 
\[
\kappa = \frac{1}{2}, \,\, A=\{2,3,4,\ldots \}, \,\, \nu(k)=k-1, \,\,\lambda_k=1, \,\, j_n=n+1, \,\, \text{and} \,\ s_n=n,
\]
we conclude that $\|S\|=1$ and that $S^*$ does not attain its norm. Lemma \ref{lem:isometric-conjugacy} now yields that $T \not\in \NA_2 (W_\alpha, c)$.

\underline{Case 2}: If $W_\alpha$ is isometric to either $c_0$ or $c$, then the same construction of $T:W_\alpha \to c$ in \eqref{eq:definition_of_T_target_c} works after replacing $\phi$ by $\mathrm{Id}\colon \ell_1\to c_0^*$ (identity map) or by $\Phi\colon \ell_1\to c^*$ from \eqref{eq:c_dual}, respectively.
\end{proof}

\subsection{The general case \texorpdfstring{$W_\beta$}{Wbeta}}\label{subsec:W_beta_target}

Having shown that the equality fails when the target space is $c$, we now turn to the remaining class of $\ell_1$-preduals. The next result shows that the same negative phenomenon persists whenever the target space is $W_\beta$ and $W_\beta$ is isometric neither to $c$ nor to $c_0$.

\begin{theorem}\label{thm:target_W_beta}
    Let $\alpha,\beta \in S_{\ell_1}$ be such that $W_\alpha^* \equiv W_\beta^* \equiv \ell_1$. Suppose that $W_\beta$ is isometric neither to $c$ nor $c_0$. Then $\NA_2 (W_\alpha, W_\beta) \neq \mathcal{L}(W_\alpha, W_\beta)$.
\end{theorem}

\begin{proof}
As in Theorem \ref{thm:target_c}, set 
\[
u= \left(\frac{1}{2^n} \right) \quad  \text{and}  \quad  b_n = \frac{1}{2} (e_{n+1}-e_{n+2}) \quad (n\in\mathbb{N}),
\]
where $(e_n)$ denotes the standard coordinate basis for $\ell_1$. 

\underline{Case 1}: $W_\alpha$ is isometric neither to $c$ nor $c_0$.

\noindent \underline{Case 1-1}: Suppose that $\beta$ is finitely supported, i.e., $\beta= (\beta_1,\ldots, \beta_N, 0,0, \ldots )$ for some $N \in \mathbb{N}$. Let $\phi_\alpha \colon \ell_1 \to W_\alpha^*$ and $\phi_\beta\colon \ell_1\to W_\beta^*$ be the canonical linear isometries (see Remark \ref{remark:onto}). Since both maps are given by the same formula, we simply write $\phi$ instead of $\phi_\alpha$ and $\phi_\beta$ whenever no confusion arises. Choose $m \in \{1,\ldots, N-1 \}$ such that $\beta_{m+1} \neq 0$. Put \[
v:= -\frac{\beta_1}{\beta_{m+1}}
\]
and choose $0<\theta<1$ such that $(1-\theta) |v| < 1$.

    Define $T \in \mathcal{L}(W_\alpha, c)$ by $T(x) = ((Tx)(n))_{n\in\mathbb{N}}$ for $x \in W_\alpha$, where 
    \begin{equation}\label{eq:definition_of_T_W_beta_case1}
     (Tx)(k) := \begin{cases}
           v(1-\theta) \,\phi(u )(x) \quad &\text{ for } k = m \\ 
          \phi( (1-\theta) u+ \theta b_k)(x) \quad &\text{ for } k \geq N \\ 
          0 &\text{ otherwise. }
     \end{cases}
    \end{equation}
    Note that $(Tx)(n) \to (1-\theta)\phi (u)(x)$ as $n\to\infty$. Moreover, 
    \begin{equation*}
    \beta_1 \lim_n (Tx)(n) + \sum_{k=1}^\infty \beta_{k+1} (Tx)(k) 
    = \beta_1 (1-\theta) \phi(u)(x)  + \beta_{m+1} (v(1-\theta)\phi(u)(x) )=0.
    \end{equation*}
    It follows that $T(x) \in W_\beta$ for every $x \in W_\alpha$. Since the coordinate functionals in \eqref{eq:definition_of_T_W_beta_case1} are uniformly bounded, we may regard $T$ as a bounded linear operator from $W_\alpha$ into $W_\beta$.

As in the proof of Theorem \ref{thm:target_c}, we can observe that 
    \[
    T^* (\phi_\beta (e_k)) = \begin{cases}
        v(1-\theta) \phi (u) \quad &\text{ for } k=m; \\ 
        \phi( (1-\theta) u+ \theta b_k) \quad &\text{ for } k \geq N \\
        0 &\text{ otherwise. }
    \end{cases} 
    \]
    Let $S:= \phi_\alpha^{-1}T^*\phi_\beta$. Then by \eqref{eq:definition_of_T_W_beta_case1}
\[
Se_k=
\begin{cases}
v(1-\theta)u, \quad &\text{ for } k=m,\\
(1-\theta)u+\theta b_k, \quad &\text{ for } k\ge N,\\
0, \quad &\text{ otherwise}.
\end{cases}
\]
Define $A=\{N,N+1,\ldots\}$. Note that $\sup_{k\notin A}\|Se_k\|=(1-\theta)|v|<1$.
Applying Lemma~\ref{lem:moving-bump} with
\[
\kappa=1-\theta, \,\, \nu(k)=k, \,\, \lambda_k=1, \,\, j_n = s_n = N+n-1,
\]
we observe that $S^*$ does not attain its norm. Hence $T \not\in \NA_2 (W_\alpha,W_\beta)$ by Lemma \ref{lem:isometric-conjugacy}.

\noindent \underline{Case 1-2}: Suppose that $\beta$ is infinitely supported. For simplicity, set $\sigma = \sum_{j=1}^\infty \beta_j$. Take $m \in \mathbb{N}$ such that $\beta_{m+1} \neq 0$. Define $\gamma \in c_{00}$ by 
    \[
    \gamma(m) = - \frac{\sigma}{\beta_{m+1}} \quad \text{ and } \quad \gamma(j) = 0 \text{ otherwise}.
    \]
    Let $\kappa \in (0,1/2)$ be sufficiently small so that 
\begin{equation}\label{eq:kappa}
\kappa |1 + \gamma(m)| < \delta < 1    
\end{equation}
for some $\delta \in (0,1)$. Since $\beta$ is infinitely supported, we can find increasing sequences $\{p_n\}$ and $\{q_n\}$ in $\mathbb{N} \setminus \{m\}$ such that all the indices $p_n$ and $q_n$ are pairwise distinct, and 
\begin{equation*}
    \beta_{p_n +1}\neq 0, \, \beta_{q_n+1} \neq 0, \, \text{ and } \, |\beta_{q_n +1}| \leq |\beta_{p_n +1} |. 
\end{equation*}
Set $r_n = \beta_{q_n +1}/\beta_{p_n +1}$; thus $|r_n| \leq 1$.
Define $T\colon W_\alpha \to c$ by 
\begin{equation}\label{eq:definition_of_T_W_beta_case2}
(Tx)(k) = \begin{cases}
        \kappa (1+\gamma(m)) \phi(u)(x)  \quad &\text{ for } k=m; \\ 
        \kappa \phi(u)(x) - (1-\kappa) r_n \phi(b_n) (x) \quad &\text{ for } k= p_n; \\
        \kappa \phi(u)(x) + (1-\kappa) \phi(b_n) (x) \quad &\text{ for } k= q_n; \\ 
        \kappa \phi(u)(x) \quad &\text{ otherwise.}  
    \end{cases} 
\end{equation}
We claim that the range of $T$ is contained in $W_\beta$. Observe first that for $x \in W_\alpha$, $\phi(b_n)(x) = (x_{n+1}-x_{n+2})/2\to 0$. It follows that $(Tx)(k) \to \kappa \phi(u)(x)$ as $k\to\infty$. Next, 
\begin{align*}
    &\beta_1 \lim_k (Tx)(k) + \sum_{k=1}^\infty \beta_{k+1} (Tx)(k) \nonumber \\
    &\quad = \sum_{k=1}^\infty \beta_{k} \kappa \phi(u)(x) + \beta_{m+1} \kappa \gamma(m) \phi(u)(x) + \sum_{n=1}^\infty (1-\kappa) ( \beta_{q_n +1} - \beta_{p_n +1} r_n  ) \phi(b_n)(x) \nonumber \\ 
    &\quad =\sigma \kappa \phi(u)(x) - \beta_{m+1}\kappa\left( \frac{\sigma}{\beta_{m+1}} \right) \phi(u)(x) +  \sum_{n=1}^\infty (1-\kappa) ( \beta_{q_n +1} - \beta_{q_n+1} ) \phi(b_n)(x) \nonumber =0. \nonumber
\end{align*}
This proves that $T(x) \in W_\beta$ for every $x \in W_\alpha$. Notice that all the coefficient functionals in \eqref{eq:definition_of_T_W_beta_case2} have norm at most one. Hence $T \in \mathcal{L}(W_\alpha,W_\beta)$.

Consider $S:=\phi_\alpha^{-1}T^*\phi_\beta$ and observe from \eqref{eq:definition_of_T_W_beta_case2} that 
\[
Se_k = \begin{cases}
        \kappa (1+\gamma(m)) u  \quad &\text{ for } k=m; \\ 
        \kappa u - (1-\kappa) r_n b_n \quad &\text{ for } k= p_n; \\
        \kappa u + (1-\kappa) b_n \quad &\text{ for } k= q_n; \\ 
        \kappa u \quad &\text{ otherwise.}  
    \end{cases} 
\]
Apply Lemma~\ref{lem:moving-bump} with $A=\{p_n,q_n \colon n\in\mathbb N\}$ and 
\[
\nu(p_n)=\nu(q_n)=n, \,\, \lambda_{p_n}=-r_n, \,\, \lambda_{q_n}=1, \,\, j_n=q_n,\,\, s_n =n.
\]
Note from \eqref{eq:kappa} that 
\[
\sup_{k \not\in A} \|Se_k\| \leq \max\{\delta, \kappa\} <1.
\]
Therefore, $S^*$ does not attain its norm, and Lemma \ref{lem:isometric-conjugacy} yields that $T \not\in \NA_2 (W_\alpha,W_\beta)$.

\underline{Case 2}: $W_\alpha$ is isometric to either $c_0$ or $c$. In this case, the same construction works with only a minor modification. Namely, in the definition of $T$ in \eqref{eq:definition_of_T_W_beta_case1} and \eqref{eq:definition_of_T_W_beta_case2}, we replace $\phi$ by the canonical identification $\mathrm{Id}\colon \ell_1\to c_0^*=\ell_1$ (identity map) when $W_\alpha\equiv c_0$, and by $\Phi\colon \ell_1\to c^*$ from \eqref{eq:c_dual} when $W_\alpha\equiv c$.
\end{proof}

\subsection{Norm-attaining second adjoints on \texorpdfstring{$\ell_1$}{l1}}\label{subsec:l1}

In this subsection, we provide the $\ell_1$ part of Theorem~\ref{thm:main_c0}. 

\begin{theorem}\label{thm:NA2_l1_l1}
    $\mathcal{L}(\ell_1,\ell_1)=\NA_2 (\ell_1,\ell_1)$.
\end{theorem}

We will need the following fact, which simply means that $c_0$ is an $M$-ideal in $c_0^{**}=\ell_\infty$ (see \cite[Example III.1.4.(a)]{HWW}).

\begin{fact}\label{fact:Hewitt-Yosida}
    $\ell_\infty^* = c_0^\perp \oplus_1 \ell_1$.
\end{fact}

\begin{proof}[Proof of Theorem~\ref{thm:NA2_l1_l1}]
    Let $T \in \mathcal{L}(\ell_1,\ell_1)$ be such that $\|T\|=1$. Denoting by $(e_n)$ the canonical basis of $\ell_1$, note that $\|T\| =\sup_n \|Te_n\|$. If the supremum is attained at some $n$, then $T \in \NA (\ell_1,\ell_1)$. Suppose that the supremum is not attained; hence there exists an increasing sequence $(n_k) \subseteq \mathbb{N}$ such that $\|Te_{n_k} \| \to 1$. Passing to a subsequence and relabeling, we may assume that $\|Te_n\| \to 1$. Put $y_n := Te_n$ for every $n \in \mathbb{N}$. 

    Passing to a further subsequence, assume that $(y_n)$ converges to some $y \in B_{\ell_1}$ in the weak$^*$ topology (which is metrizable on bounded subsets of $\ell_1$). Let $u_n : = y_n -y \in 2B_{\ell_1} $ for every $n \in \mathbb{N}$. Then $(u_n)$ is weak$^*$-null. By applying Lemma \ref{lemma:disjoint_support_lemma}, we can find a bounded sequence $(b_n)$ in ${\ell_1}$ such that 
\begin{enumerate}[label=(\roman*)] 
\itemsep0.25em
\item $\supp(b_n) \cap \supp(b_m) = \varnothing$ for every $n \not= m$ and 
		\item there exists an increasing sequence $(\tau(n))$ of natural numbers such that $\|u_{\tau(n)} - b_{n}\| \to 0$ as $n \to \infty$.
        \end{enumerate}
We may also assume that $F_n:=\supp(b_n)$ is finite for every $n\in\mathbb N$.

        Note that $\| y \restricted_{F_n} \| \to 0$ since $(F_n)$ is pairwise disjoint. Observe also that 
        \begin{align*}
            \|y\| + \|b_n\| - \|y+b_n\| &= \|y \restricted_{F_n}\|  + \|b_n\| - \|y \restricted_{F_n} +b_n\| \\
            &\leq \|y \restricted_{F_n}\|  + \|b_n\| - (  \|b_n\| - \|y \restricted_{F_n}\|  ) = 2 \|y \restricted_{F_n} \|.
        \end{align*}
        It follows that 
        \begin{equation}\label{eq:y_bn_estimate}
        \bigl|  \|y\| + \|b_n\| - \|y+b_n\| \bigr| \leq 2 \|y \restricted_{F_n} \| \to 0.
        \end{equation}
        On the other hand, 
        \begin{equation}\label{eq:y_bn_estimate2}
        \bigl| \| y + b_n\| - \|y_{\tau(n)} \| \bigl| = \bigl| \| y + b_n\| - \|y+u_{\tau(n)} \| \bigr| \leq \|b_n - u_{\tau(n)}\| \to 0 
        \end{equation}
        Combining \eqref{eq:y_bn_estimate} and \eqref{eq:y_bn_estimate2} with the fact that $\|y_{\tau(n)}\|\to1$, we obtain that
        \begin{equation}\label{eq:limit_of_bn}
        \|b_n\| \to 1 - \|y\|.
        \end{equation}  
        Let $\mathcal{U}$ be a free ultrafilter on $\mathbb{N}$, and define $\phi \colon \ell_\infty \to \mathbb{R}$ by 
        \[
        \phi (\zeta) = \lim_{\mathcal{U}} \zeta(e_{\tau(n)}) \quad (\zeta \in \ell_\infty).
        \]
        Then $\|\phi\| = 1$, and for $\zeta \in \ell_\infty$, we have 
        \begin{align*}
            (T^{**} \phi)(\zeta) = \phi (T^* \zeta) &= \lim_{\mathcal U} \zeta(Te_{\tau(n)}) \\
            &=\lim_{\mathcal U} \zeta(y_{\tau(n)}) = \lim_{\mathcal{U}} \zeta(u_{\tau(n)} + y ) = \zeta(y) + \lim_{\mathcal{U}} \zeta(b_n).
        \end{align*}
       If we denote by $\psi\colon \ell_\infty\to\mathbb R$ the functional given by
       \[
       \psi(\zeta)=\lim_{\mathcal{U}} \zeta(b_n),
       \]
       then we have  
        \[
        T^{**} \phi = y + \psi,
        \]
        where $y \in \ell_1 \hookrightarrow \ell_\infty^*$. Since $(b_n)$ have pairwise disjoint supports, $\psi$ annihilates $c_0$, i.e., $\psi \in c_0^\perp$. Moreover, if $z_0 \in S_{\ell_\infty}$ is selected in such a way that $z_0 (b_n)=\|b_n\|$ for every $n \in \mathbb{N}$ (which is possible because the supports of $(b_n)$ are pairwise disjoint), then
        \[
        \psi (z_0) = \lim_{\mathcal U} z_0(b_n) = \lim_{\mathcal U} \|b_n\| = 1 - \|y\|
        \]
        where the last equality follows from \eqref{eq:limit_of_bn}. By definition, $\|\psi\| \leq \lim_{\mathcal U} \|b_n\| = 1 - \|y\|$; hence $\|\psi\| = 1-\|y\|$. 
        Therefore, by Fact~\ref{fact:Hewitt-Yosida},
        \[
        \|T^{**} \phi \| = \|y\| + \|\psi\| = \|y\| + (1-\|y\|)=1. 
        \]
        This proves that $T \in \NA_2 (\ell_1,\ell_1)$, as desired.
\end{proof}

Observe that this result, together with Theorem~\ref{thm:target_c}, provides the following interesting example.

\begin{corollary}\label{cor:ccN3-noN2}
$\NA_3(c,c)=\mathcal L(c,c)$ while $\NA_2(c,c)\neq \mathcal L(c,c)$.    
\end{corollary}

Since $c$ is isomorphic to $c_0$, Theorem \ref{theorem:c0} and Corollary \ref{cor:ccN3-noN2} already show that the property 
\[
\NA_2 (X,X)=\mathcal{L}(X,X)
\]
is not an isomorphic property. Theorem \ref{Ostrovskii-for-NA1-and-NA2} will strengthen this observation by showing that the property can be destroyed by an equivalent renorming on every infinite-dimensional Banach space.

\begin{remark}\label{rem-ell1gamma}
The proof of Theorem \ref{thm:NA2_l1_l1} routinely extends to spaces $\ell_1(\Gamma)$. Let $\Gamma$ be an arbitrary (infinite) set. Then  $$\mathcal{L}(\ell_1(\Gamma), \ell_1(\Gamma))= \NA_2(\ell_1(\Gamma), \ell_1(\Gamma)).$$ 
\end{remark}

 \section{Holub--Mujica-type theorems for \texorpdfstring{$\NA_1$}{NA1} and \texorpdfstring{$\NA_2$}{NA2}}\label{section:HolubMujica}

In this section, we establish Holub--Mujica-type theorems for the classes
$\NA_1$ and $\NA_2$. Our proofs are inspired by the proof of \cite[Theorem B]{DJM} (Theorem \ref{thm:DJM}), but with two important differences. For $\NA_1$, the natural ambient space is the subspace $\mathcal{L}_{w^*,w^*} (Y^*, X^*)$ of $\mathcal{L}(Y^*,X^*)$ consisting of weak$^*$-weak$^*$-continuous operators, rather than $\mathcal{L}(X,Y)$, and the Principle of Local Reflexivity enters through Lemma~\ref{lem:PLR_Johnson} to pass from finite-rank operators to weak$^*$-weak$^*$-continuous ones. For $\NA_2$, the ambient space is the subspace $\mathcal{L}_{w^*,w^*}^{(2)}(X, Y)$ of $\mathcal{L}(X^{**},  Y^{**})$ consisting of operators of the form $S^{**}$ for $S \in \mathcal{L}(X,Y)$, and a version of the Principle of Local Reflexivity (Theorem~\ref{thm:PLR}) is again required. In both cases, the separability assumption plays a crucial role in promoting sequential closure to full closure (Lemma~\ref{lemmaC-NA1} and Lemma~\ref{lem:NA2-SOTseqclosed}).

\begin{theorem}[\text{\cite[Theorem B]{DJM}}]\label{thm:DJM} Let $X$ be a reflexive space and $Y$ be an arbitrary Banach space. Consider the following statements.
\begin{enumerate}[label=(\alph*)]
\itemsep0.25em
		\item $\mathcal{K}(X, Y) = \mathcal{L}(X,Y)$.
		\item $\NA(X, Y) = \mathcal{L}(X,Y)$.
		\item The unit ball $B_{\mathcal{K}(X,Y)}$ is SOT-closed.
		\item $(\mathcal{L}(X,Y), \tau_c)^* = (\mathcal{L}(X,Y), \|\cdot\|)^*$.
	\end{enumerate}
Then, the implications (a) $\Rightarrow$ (b) $\Rightarrow$ (c) and (a) $\Rightarrow$ (d) $\Rightarrow$ (c) hold. If, moreover, the pair $(X,Y)$ has the BCAP, then (c) $\Rightarrow$ (a) and hence the four statements are equivalent.
\end{theorem}

\subsection{The case of \texorpdfstring{$\NA_1$}{NA1}}
A key ingredient in the proof of Theorem \ref{thm:DJM} is the property $(P)$ introduced in \cite{DJM}. Since our setting involves the space $\mathcal L_{w^*,w^*}(Y^*,X^*)$ rather than $\mathcal L(X,Y)$, we need a corresponding version of this property adapted to the weak$^*$-to-weak$^*$ continuous framework.

\begin{definition} \label{definition-property-star} Let $X$ and $Y$ be Banach spaces. The pair $(X,Y)$ is said to have \textit{property $\propertystar$} if there is a relatively $\wot$ compact set $K \subseteq \weakstar$ such that 
\begin{equation*}
0 \in \overline{K}^{\wot} \ \ \ \mbox{and} \ \ \ 0 \not\in \overline{\co}^{\|\cdot\|}(K).
\end{equation*}
\end{definition}

The main result of this subsection is the following theorem.

\begin{theorem} \label{theorem-holub-mujica-NA1}
	Let $X$ and $Y$ be Banach spaces. Suppose that $Y^*$ is separable, and that the pair $(Y^*,X^*)$ has the $\lambda$-BAP. If $\NA_1 (X,Y)=\mathcal{L}(X,Y)$, then $\mathcal{K}(X,Y)=\mathcal{L}(X,Y)$.
\end{theorem}

To prove Theorem \ref{theorem-holub-mujica-NA1}, we proceed through a sequence of lemmas.

\begin{lemma} \label{lemmaA-NA1} Let $X$ and $Y$ be Banach spaces. If $\NA_1(X,Y) = \mathcal{L}(X,Y)$ then $(X,Y)$ fails property $\propertystar$.
\end{lemma}

\begin{proof} Suppose that $\NA_1(X,Y) = \mathcal{L}(X,Y)$. Then, the set 
\begin{equation*}
B : = \{y^* \otimes x^{**}\colon y^* \in S_{Y^*}, x^{**} \in S_{X^{**}}\}
\end{equation*}
is a James boundary for $\weakstar$. 

Notice that, on $\weakstar$, the topology $\sigma(\weakstar, B)$ coincides with the WOT topology. Suppose, for contradiction, that there exists $K \subseteq \weakstar$ such that 
\begin{equation*}
0 \in \overline{K}^{\wot} \ \ \ \mbox{and} \ \ \ 0 \not\in \overline{\co}^{\|\cdot\|}(K).
\end{equation*}
Since 
\begin{equation*}
\overline{K}^{\mathrm {WOT}} = \overline{K}^{\sigma(\weakstar, B)}
\end{equation*}
is WOT compact, by a theorem of Pfitzner (see \cite{Pf} or \cite[Theorem 3.121]{FHHMZ}), $\overline{K}^{\mathrm {WOT}}$ is weakly compact. Thus, the weak and WOT topologies coincide on $\overline{K}^{\mathrm{WOT}}$. Hence $0 \in \overline{K}^{\mathrm {WOT}} \subseteq \overline{K}^w$ which implies that $0 \in \overline{\co}^{\|\cdot\|}(K)$, a contradiction.
\end{proof}

The failure of property $(P^\star)$ has important consequences for the closure behavior of convex subsets of $\mathcal L_{w^*,w^*}(Y^*,X^*)$.

\begin{lemma} \label{lemmaB-NA1} Let $X$ and $Y$ be Banach spaces. If $(X, Y)$ fails property $\propertystar$, then every norm-closed convex set $C \subseteq \weakstar$ is WOT sequentially closed in $\weakstar$. That is, 
\[
\overline{C}^{\mathrm{WOT-seq}} \cap \weakstar = C.
\]
\end{lemma}

\begin{proof} Suppose that $(R_n) \subseteq C$ and $R \in \weakstar$ such that $R_n \xrightarrow{\mathrm {WOT}} R$, but $R \not\in C$. Define 
	\begin{equation*}
	K:= \{ R_n - R\colon n \in \N\}.
	\end{equation*}
Then, $K$ is a relatively WOT compact subset of $\weakstar$ and $0 \in \overline{K}^{\mathrm{WOT}}$. However, if $0 \in \overline{\co}^{\|\cdot\|}(K)$, then this would imply that $R \in \overline{\co}^{\|\cdot\|} \{R_n\colon n \in \N\} \subseteq C$, which is a contradiction.
\end{proof}

To upgrade sequential closure to full closure, we use the separability assumption on $Y^*$. To this end, we state and briefly prove the following general fact. 

\begin{lemma}\label{lem:general_metrizable_SOT}
    Let $E$ be a separable Banach space, let $F$ be a Banach
space, and let $\mathscr M$ be a linear subspace of $\mathcal L(E,F)$.
On every norm-bounded subset of $\mathscr M$, the strong operator
topology is metrizable. Consequently, relative SOT sequential
closedness and relative SOT closedness coincide on bounded subsets
of $\mathscr M$.
\end{lemma}

\begin{proof}
    Let $(x_n)$ be dense in $B_E$. On a norm-bounded subset of $\mathscr M$, the metric 
    \[
    d(S,T) := \sum_{n=1}^\infty \frac{\min\{1,\|(S-T)x_n\|\}}{2^n}
    \]
    induces the strong operator topology. 
\end{proof}

\begin{lemma} \label{lemmaC-NA1} Let $X, Y$ be Banach spaces and assume that $Y^*$ is separable. If a bounded subset $C \subseteq \weakstar$ is SOT sequentially closed in $\weakstar$, then $C$ is SOT closed in $\weakstar$.
\end{lemma}

\begin{proof}
    Apply Lemma \ref{lem:general_metrizable_SOT} to $\mathscr M = \weakstar \subseteq \mathcal{L}(Y^*, X^*)$ and $E=Y^*$.
\end{proof}

The remaining ingredient is a weak$^*$ version of the Principle of Local Reflexivity, which allows finite-rank operators to be approximated by weak$^*$-weak$^*$ continuous finite-rank operators.

\begin{lemma}[\text{\cite[Lemma 1]{Johnson}}]\label{lem:PLR_Johnson}
    Let $X$ be a Banach space and $F$ be a finite-dimensional Banach space. If $A \subset X^*$ is a finite set, $S\colon X^* \to F$ is a bounded linear operator, and $\varepsilon>0$, then there exists a weak$^*$ to weak$^*$ continuous linear operator $T\colon X^* \to F$ such that $Tx^* =Sx^* $ for all $x^* \in A$ and $\|T\| \leq \|S\|+\varepsilon$.
\end{lemma}

\begin{lemma} \label{Holub-NA1-Lemma5} Let $X$ and $Y$ be Banach spaces. Then,
\begin{equation*}
    B_{\mathcal{F}(Y^*, X^*)} \subseteq \overline{B_{\weakstarfiniterank}}^{\tau_c}. 
\end{equation*}
\end{lemma}

\begin{proof} Let $T \in \mathcal{F}(Y^*, X^*)$ with $\|T\| \leq 1$ be given. Let $K \subseteq Y^*$ be compact and let $\eps > 0$ be given. Choose $\delta > 0$ sufficiently small. Take $A \subseteq K$ to be a finite $\delta$-net. By the Principle of Local Reflexivity (Lemma \ref{lem:PLR_Johnson}), there is a $w^*$-$w^*$ continuous $R\colon Y^* \to T(Y^*) \subseteq X^*$ such that $R\restricted_A = T\restricted_A$ and $\|R\| \leq \|T\| + \delta \leq 1 + \delta$. Notice now that $(1+\delta)^{-1} R \in B_{\weakstarfiniterank}$ and $(1 + \delta)^{-1}R$ is arbitrarily close to $T$ on $K$. This finishes the proof.
\end{proof}
	
	We now have all the ingredients needed for the proof of Theorem \ref{theorem-holub-mujica-NA1}. 
    
\begin{proof}[Proof of Theorem \ref{theorem-holub-mujica-NA1}] Assume that $Y^*$ is separable, that the pair $(Y^*,X^*)$ has the $\lambda$-BAP and also that $\NA_1(X, Y) = \mathcal{L}(X,Y)$. By Lemma \ref{lemmaA-NA1}, $(X,Y)$ fails property $\propertystar$. It follows from Lemma \ref{lemmaB-NA1} and Lemma \ref{lemmaC-NA1} that $B_{\mathcal{K}_{w^*,w^*} (Y^*,X^*)}$ is SOT closed in $\weakstar$, i.e., 
\begin{equation}\label{eq:chain-1}
\overline{B_{\mathcal{K}_{w^*,w^*} (Y^*,X^*)}}^{\mathrm{SOT}} \cap \weakstar = B_{\mathcal{K}_{w^*,w^*} (Y^*,X^*)}.    
\end{equation}
 Since the pair $(Y^*,X^*)$ has the $\lambda$-BAP, we have
	\begin{equation}\label{eq:chain0}
 B_{ \mathcal{L} (Y^*,X^*) } \subseteq \lambda \overline{ B_{\mathcal{F} (Y^*,X^*)} }^{\tau_c}.
 \end{equation}
 It follows from the Principle of Local Reflexivity (Lemma \ref{Holub-NA1-Lemma5}) that 
	\begin{equation}\label{eq:chain1}
	 \lambda \overline{ B_{\mathcal{F} (Y^*,X^*)} }^{\tau_c} \subseteq \lambda \overline{ B_{\mathcal{F}_{w^*,w^*} (Y^*,X^*)} }^{\tau_c} \subseteq \lambda \overline{ B_{\mathcal{K}_{w^*,w^*} (Y^*,X^*)} }^{\tau_c}.
	\end{equation}
    Combining \eqref{eq:chain-1}, \eqref{eq:chain0}, and \eqref{eq:chain1}, we have 
	\begin{align*}
	B_{\weakstar} &\subseteq \lambda \overline{ B_{\mathcal{K}_{w^*,w^*} (Y^*,X^*)} }^{\tau_c}  \cap \weakstar \\
    &\subseteq   \lambda \overline{ B_{\mathcal{K}_{w^*,w^*} (Y^*,X^*)} }^{\mathrm{SOT}} \cap \weakstar = \lambda B_{\mathcal{K}_{w^*,w^*} (Y^*,X^*)}.     
	\end{align*}
	This proves that $\mathcal{L}(X,Y)=\mathcal{K}(X,Y)$.
\end{proof}

\subsection{The case of \texorpdfstring{$\NA_2$}{NA2}}

 The purpose of this subsection is to prove an $\NA_2$-version of Theorem \ref{theorem-holub-mujica-NA1}. The conclusion is not completely analogous to the $\NA_1$ case: while Theorem \ref{theorem-holub-mujica-NA1} forces all operators from $X$ into $Y$ to be compact, the second-adjoint argument yields compactness only for weakly compact operators. This difference comes from the role of the second dual. The relevant ambient space is now the space of second adjoints and Oja's principle of local reflexivity (see Theorem \ref{thm:PLR} below) allows us to approximate finite-rank operators from $X^{**}$ into $Y$, but not arbitrary finite-rank operators from $X^{**}$ into $Y^{**}$. Thus the argument applies precisely to those operators $S\colon X \rightarrow Y$ satisfying $S^{**}(X^{**}) \subseteq Y$, that is, to weakly compact operators. With this modification, the proof follows the same general strategy as in the previous subsection.

\begin{theorem} \label{Holub-for-NA2} Let $X$ and $Y$ be Banach spaces, and suppose that $X^{**}$ is separable, and that the pair $(X^{**},Y)$ has the $\lambda$-BAP. If $\NA_2(X, Y) = \mathcal{L}(X, Y)$, then $\mathcal{W}(X,Y)=\mathcal{K}(X,Y)$.
\end{theorem}

The main additional ingredient, compared with the proof of Theorem \ref{theorem-holub-mujica-NA1}, is a version of the Principle of Local Reflexivity adapted to second adjoints, due to Oja \cite[Theorems 1.2 and 1.3]{Oja}. We recall it here for the convenience of the reader.

\begin{theorem}[\text{Oja's Principle of Local Reflexivity \cite[Theorem 1.2 and Theorem 1.3]{Oja}}]\label{thm:PLR} Let $X$ and $Y$ be Banach spaces, and let $S \in \mathcal{F}(X^{**},Y^{**})$. If $E$ and $F$ are finite-dimensional subspaces of $X^{**}$ and $Y^*$, respectively, and $\varepsilon>0$, then there exists $T \in \mathcal{F}(X,Y)$ such that 
	\begin{enumerate}
		\itemsep0.25em
		\item $|\|T\|-\|S\||<\varepsilon$;
		\item $x^{**} (T^* y^*) = (Sx^{**} )(y^*)$ for all $x^{**} \in E$ and $y^* \in F$;
		\item $T^{**} x^{**} = S x^{**}$ for all $x^{**} \in E$ with $Sx^{**} \in Y$.
	\end{enumerate}
\end{theorem}

As a consequence of Theorem \ref{thm:PLR}, finite-rank operators from $X^{**}$ into $Y$ can be approximated, in the topology $\tau_c$, by second adjoints of finite-rank operators on $X$.

\begin{corollary}\label{cor:PLR}
	Let $X$ and $Y$ be Banach spaces. Then 
	\[
	B_{\mathcal{F}(X^{**}, Y)} \subseteq  \overline{\{T^{**}\colon T \in B_{\mathcal{F}(X,Y)}  \}}^{\tau_c}.
	\]
\end{corollary}

\begin{proof} Let $S \in B_{\mathcal{F}(X^{**}, Y)}$, let $K \subseteq X^{**}$ be compact, and let $\eps > 0$ be given. Let $A \subseteq K$ be a finite $\eps$-net. Applying Theorem \ref{thm:PLR} with $E=\text{span}(A)$, we have that there exists $T \in \mathcal{F}(X,Y)$ such that $| \|T\| - \|S\| | < \eps$ and $T^{**}\restricted_A = S\restricted_A$. For every $x^{**} \in K$, there exists $a \in A \cap B(x^{**}, \eps)$ such that 
	\begin{align*}
		\|T^{**} x^{**} - S x^{**}\| &\leq \|T^{**} x^{**} - T^{**} a\| + \|T^{**}a - Sa\| + \|Sa - Sx^{**}\| \\
		&\leq \|x^{**} - a\| \|T^{**}\| + \|a-x^{**}\| \|S\| < \eps(1 + \eps) + \eps = \eps^2 + 2 \eps.
	\end{align*}
	Thus, considering the operator $R:= T/ (1+\varepsilon) \in B_{\mathcal{F}(X,Y)}$ finishes the proof.
\end{proof}

Let $X$ and $Y$ be Banach spaces. For simplicity, we use the following notation.
\begin{equation*}
\mathcal{L}_{w^*,w^*}^{(2)}(X, Y) := \{ T \in \mathcal{L}(X^{**}, Y^{**})\colon \exists S \in \mathcal{L}(X,Y) \ \mbox{such that} \ S^{**} = T \}.
\end{equation*}
As in the proof of Theorem \ref{theorem-holub-mujica-NA1}, a suitable analogue of property $(P)$ will play a central role. Since the natural ambient space in the present setting is $\mathcal{L}_{w^*,w^*}^{(2)}(X, Y)$, the space formed by second adjoints, we introduce the corresponding version of property $(P^\star)$.

\begin{definition} Let $X$ and $Y$ be Banach spaces. The pair $(X, Y)$ is said to have \textit{property $\propertystarstar$} if there is a relatively WOT compact $K \subseteq  \mathcal{L}_{w^*,w^*}^{(2)}(X, Y)$ such that 
\begin{equation*}
    0 \in \overline{K}^{\mathrm{WOT}} \ \ \ \ \ \mbox{and} \ \ \ \ \ 0 \not\in \overline{\co}^{\|\cdot\|}(K).
\end{equation*}
\end{definition}

The following lemma is the corresponding version of Lemma \ref{lemmaA-NA1} for property $\propertystarstar$.

\begin{lemma}\label{lem:NA2-propertystar} Let $X$ and $Y$ be Banach spaces. If $\NA_2(X,Y) = \mathcal{L}(X,Y)$, then $(X,Y)$ fails property $\propertystarstar$.
\end{lemma}

\begin{proof} Suppose that $\NA_2(X,Y) = \mathcal{L}(X,Y)$. The set 
\begin{equation*}
    B:= \{x^{**} \otimes y^{***}\colon x^{**} \in S_{X^{**}}, y^{***} \in S_{Y^{***}} \} 
\end{equation*}
is then a James boundary for $\mathcal{L}_{w^*,w^*}^{(2)}(X, Y)$. Notice that $\sigma\left(\mathcal{L}_{w^*,w^*}^{(2)}(X, Y), B\right)$ is the WOT topology on $\mathcal{L}_{w^*,w^*}^{(2)}(X, Y)$. Suppose, for contradiction, that there is $K \subseteq \mathcal{L}_{w^*,w^*}^{(2)}(X, Y)$ satisfying property $(P^{**})$. Since $K \subseteq \mathcal{L}_{w^*,w^*}^{(2)}(X, Y)$ is relatively WOT compact, we have that 
\begin{equation*}
    \overline{K}^{\mathrm{WOT}} = \overline{K}^{\sigma\left( \mathcal{L}_{w^*,w^*}^{(2)}(X, Y), B\right)}
\end{equation*}
is WOT compact, which in turn implies that it is $\sigma\left(\mathcal{L}_{w^*,w^*}^{(2)}(X, Y), B\right)$ compact. By Pfitzner's theorem, $\overline{K}^{\mathrm{WOT}}$ is weakly compact, that is, $\overline{\co}^{w}(K)$ is $w$-compact and 
\begin{equation*}
    0 \in \overline{K}^{\mathrm{WOT}} \subseteq \overline{\co}^{w}(K) = \overline{\co}^{\|\cdot\|}(K) 
\end{equation*}
which is a contradiction.
\end{proof}

As in Lemma \ref{lemmaB-NA1}, the failure of property $\propertystarstar$ leads to a sequential closure principle. The next lemma provides the corresponding result in the present setting.

\begin{lemma}\label{lem:NA2-WOTseq} Let $X$ and $Y$ be Banach spaces. If $(X,Y)$ fails property $\propertystarstar$, then every norm-closed convex set $C \subseteq \mathcal{L}_{w^*,w^*}^{(2)}(X, Y)$ is WOT sequentially closed in $\mathcal{L}_{w^*,w^*}^{(2)}(X, Y)$. That is, 
\[
\overline{C}^{\mathrm{WOT-seq}} \cap \mathcal{L}_{w^*,w^*}^{(2)}(X, Y) = C.
\]
\end{lemma}

\begin{proof} Suppose that there exist $(R_n) \subseteq C$ and $R \in \mathcal{L}_{w^*,w^*}^{(2)}(X, Y)$ such that $R_n \xrightarrow{\mathrm{WOT}} R$ and $R \not\in C$. Put $K := \{R_n - R\colon n \in \N \}$ which is a relatively WOT compact subset of $\mathcal{L}_{w^*,w^*}^{(2)}(X, Y)$ and satisfies $0 \in \overline{K}^{\mathrm{WOT}}$. If $0 \in \overline{\co}^{\|\cdot\|}(K)$, then $R \in  \overline{\co}^{\|\cdot\|} \{R_n\colon n \in \N\} \subseteq C$, which is a contradiction.
\end{proof}

To obtain the closure property needed later, we use the following lemma, which exploits the separability of $X^{**}$ to pass from sequential closure to full closure. This is again a consequence of Lemma \ref{lem:general_metrizable_SOT}.

\begin{lemma}\label{lem:NA2-SOTseqclosed}
    Let $X$ and $Y$ be Banach spaces. If $X^{**}$ is separable and a bounded subset $C \subseteq \mathcal{L}_{w^*,w^*}^{(2)}(X, Y)$ is SOT sequentially closed in $\mathcal{L}_{w^*,w^*}^{(2)}(X, Y)$, then $C$ is SOT closed in $\mathcal{L}_{w^*,w^*}^{(2)}(X, Y)$.
\end{lemma}

\begin{proof}
    Apply Lemma \ref{lem:general_metrizable_SOT} to $\mathscr M = \mathcal{L}_{w^*,w^*}^{(2)}(X, Y) \subseteq \mathcal{L}(X^{**}, Y^{**})$ and $E=X^{**}$.
\end{proof}

\begin{proof}[Proof of Theorem \ref{Holub-for-NA2}] Suppose that $X^{**}$ is separable and $\NA_2(X, Y) = \mathcal{L}(X,Y)$. By Lemma \ref{lem:NA2-propertystar}, $(X,Y)$ fails property $\propertystarstar$. Using the assumption that the pair $(X^{**},Y)$ has the $\lambda$-BAP, and Corollary \ref{cor:PLR}, 
\begin{equation*}
    B_{\mathcal{L}(X^{**}, Y)} \subseteq \lambda \overline{B_{\mathcal{F}(X^{**},Y)}}^{\tau_c} \subseteq \lambda \overline{\{T^{**}\colon T \in B_{\mathcal{F}(X,Y)}\}}^{\tau_c}.
\end{equation*}

    Note that 
    \[
    \overline{\{T^{**}\colon T \in  B_{\mathcal{F}(X,Y)}  \}}^{\tau_c} \subseteq \overline{\{T^{**}\colon T \in  B_{\mathcal{K}(X,Y)}  \}}^{\tau_c} \subseteq \overline{\{T^{**}\colon T \in  B_{\mathcal{K}(X,Y)}  \}}^{\mathrm{SOT}}.
    \]
It follows from Lemma \ref{lem:NA2-WOTseq} and Lemma \ref{lem:NA2-SOTseqclosed} that 
\begin{align*}
    \{ S^{**} \colon S \in B_{\mathcal{W}(X,Y)} \}  
    &=B_{\mathcal{L}(X^{**}, Y)} \cap \mathcal{L}_{w^*,w^*}^{(2)}(X, Y) \\
    & \subseteq \lambda \overline{\{T^{**}\colon T \in  B_{\mathcal{K}(X,Y)}  \} }^{\mathrm{SOT}} \cap \mathcal{L}_{w^*,w^*}^{(2)}(X, Y) = \lambda \{T^{**}\colon T \in  B_{\mathcal{K}(X,Y)}  \} .
\end{align*}

It follows that $\mathcal{W}(X,Y)=\mathcal{K}(X,Y)$.
\end{proof}

In the case $Y = c_0$ in Theorem \ref{Holub-for-NA2}, we obtain a stronger result. 

\begin{proposition}\label{prop:NA2_when_target_is_c0}
    Let $X$ be a Banach space, and suppose that $X^{**}$ is separable. If $\NA_2 (X,c_0) = \mathcal{L}(X,c_0)$, then $X$ is finite-dimensional. 
\end{proposition}

\begin{proof}
    Observe from Corollary \ref{cor:Schur} that $X^*$ has the Schur property. Since $X^{**}$ is separable, if $x^{***} \in X^{***}$, then there is a sequence $(x_n^*)$ in $X^*$ that converges to $x^{***}$ in the weak$^*$ topology. In particular, the sequence $(x_n^*)$ is a weakly Cauchy sequence in $X^*$. Since $X^*$ has the Schur property, $(x_n^*)$ is norm Cauchy. It follows that $x^{***}$ belongs to $X^*$. This shows that every element of $X^{***}$ belongs to $X^*$, so $X^*$ is reflexive. Since every reflexive Schur space is finite-dimensional, it follows that $X^*$ is finite-dimensional, and consequently, $X$ is finite-dimensional. 
\end{proof}

\begin{remark}
The pair $(c_0, c_0)$ shows that, without the separability assumption on $X^{**}$, one cannot in general strengthen the conclusion of Theorem \ref{Holub-for-NA2} to $\mathcal{L}(X,Y) = \mathcal{K}(X, Y)$ and that Proposition~\ref{prop:NA2_when_target_is_c0} fails. Indeed, $\NA_2(c_0, c_0) = \mathcal{L}(c_0, c_0)$ by Theorem \ref{theorem:c0}.
\end{remark}

We can deduce the following general result from Proposition \ref{prop:NA2_when_target_is_c0} using Lemma \ref{lem:heredity}.

\begin{corollary}
    Let $X$ be a Banach space, and suppose that $X^{**}$ is separable. If a Banach space $Y$ contains an isometric copy of $c_0$ and $\NA_2 (X,Y) = \mathcal{L}(X,Y)$, then $X$ is finite-dimensional.
\end{corollary}

When the range space is weakly sequentially complete, Theorem \ref{Holub-for-NA2} yields the following consequence.

\begin{corollary}
Let $X$ and $Y$ be Banach spaces. Suppose that $X^{**}$ is separable, that $Y$ is weakly sequentially complete, and that the pair $(X^{**},Y)$ has the $\lambda$-BAP. If $\NA_2(X, Y) = \mathcal{L}(X, Y)$, then $\mathcal{L}(X,Y) = \mathcal{K}(X,Y)$.
\end{corollary}

\begin{proof}
    Since $X^{**}$ is separable, $X$ does not contain an isomorphic copy of $\ell_1$. It follows from Rosenthal's $\ell_1$-theorem that every bounded sequence in $X$ admits a weakly Cauchy subsequence. Therefore, since $Y$ is weakly sequentially complete, every bounded linear operator $T\colon X \to Y$ is weakly compact, i.e., $\mathcal{L}(X,Y)=\mathcal{W}(X,Y)$. Now, Theorem \ref{Holub-for-NA2} finishes the proof.
\end{proof}

\section{Renorming instability of second-adjoint norm attainment} \label{sec:Ostrovskii}

Ostrovskii \cite{Ostrovskii} proved that every infinite-dimensional Banach space $X$ admits an equivalent norm, denoted by $\vertiii{\cdot}$, under which some operator in $\mathcal{L}((X, \vertiii{\cdot}), (X, \vertiii{\cdot}))$ fails to attain its norm. In this section, we show that an analogous phenomenon occurs for $\NA_2$: one can find a renorming and an operator whose second adjoint fails to attain its norm.

\begin{theorem} \label{Ostrovskii-for-NA1-and-NA2} Let $X$ be an infinite-dimensional Banach space. Then, there exist an equivalent norm $\vertiii{\cdot}$ on $X$ and a projection $P \in \mathcal{L}((X, \vertiii{\cdot}), (X, \vertiii{\cdot}))$ such that $P \not\in \NA_2((X, \vertiii{\cdot}), (X, \vertiii{\cdot}))$.
\end{theorem}

To prove Theorem~\ref{Ostrovskii-for-NA1-and-NA2}, we adapt Ostrovskii’s construction from \cite[Theorem~2]{Ostrovskii}. By Mazur's basic sequence theorem (see \cite[Theorem 1.4.5] {AlbiacKalton}, for instance) let $(e_i)_{i = 0}^{\infty}$ be a normalized basic sequence in $X$. Let $(e_i^*)_{i = 0}^{\infty} \subseteq X^*$ be Hahn-Banach extensions of its coordinate functionals such that
\begin{equation*}
\sup_{j \in \N \cup \{0\}} \|e_j^*\| < \infty.
\end{equation*}
Consider the following equivalent norm on $X$:
\begin{equation*}
\|x\|_1 := \max \left\{ \frac{1}{2} \|x\|,
\sup\nolimits_{j\in\mathbb{N}\cup\{0\}} |e_j^*(x)| \right\}
\end{equation*}
 for every $x \in X$ and set
\begin{equation*}
L := \ker e_0^* \subseteq X \ \ \ \text{and} \ \ \ B := \{ x \in L \colon \|x\|_1 \leq 1 \}.
\end{equation*}
Let $(\alpha_i)_{i = 1}^{\infty} \subseteq \R$ be such that $1 < \alpha_i < \alpha_{i + 1} < 2$ and $\lim_i \alpha_i = 2$, and let
\begin{equation*}
U := \overline{\co}^{\|\cdot\|_1} \left( B \cup \{ e_0 + \alpha_i e_i \}_{i = 1}^{\infty} \cup \{ -e_0 - \alpha_i e_i \}_{i = 1}^{\infty} \right).
\end{equation*}
The set $U$ is bounded. Moreover,
\begin{equation*}
X = L \oplus \spann \{ e_0 + \alpha_1 e_1 \} \ \ \ \mbox{and} \ \ \ \{ l + t (e_0 + \alpha_1 e_1) \colon
\|l\|_1 + |t| \leq 1 \} \subseteq U.
\end{equation*}
Consequently, $U$ is the closed unit ball of an equivalent norm $\vertiii{\cdot}$ on $X$. For every $i \geq 1$, we have $e_i \in B$, and the functional $e_i^* + (1 - \alpha_i) e_0^*$ has modulus at most one on
\begin{equation*}
B \cup \{ e_0 + \alpha_j e_j \}_{j = 1}^{\infty}
\cup \{ -e_0 - \alpha_j e_j \}_{j = 1}^{\infty}.
\end{equation*}
It follows that, for every $i \geq 1$,
\begin{equation} \label{eq:norm-of-ei}
\vertiii{e_i} = 1.
\end{equation}
Let $P := \id - e_0^* \otimes e_0$. Since $P(B) = B$ and $P(e_0 + \alpha_i e_i) = \alpha_i e_i \in 2U$, we have $P(U) \subseteq 2U$. On the other hand, by \eqref{eq:norm-of-ei}, $\vertiii{ P(e_0 + \alpha_i e_i) } = \alpha_i \to 2$. Therefore,
\begin{equation*} 
\|P\| = \|P^{**}\| = 2.
\end{equation*}

\begin{proof}[Proof of Theorem~\ref{Ostrovskii-for-NA1-and-NA2}] Suppose, toward a contradiction, that $\vertiii{ P^{**} x_0^{**} } = 2$ for  $x_0^{**} \in S_{(X, \vertiii{\cdot})^{**}}$. By the definition of $U$ and Goldstine's theorem, there is a net $(z_\lambda)_{\lambda \in \Lambda}$ in
the convex hull defined by $\co \left( B \cup \{ e_0 + \alpha_i e_i \}_{i = 1}^{\infty}
\cup \{ -e_0 - \alpha_i e_i \}_{i = 1}^{\infty} \right)$ such that
\begin{equation*} \label{eq:zlambda-to-x0}
z_\lambda \xrightarrow{\sigma(X^{**}, X^*)} x_0^{**}.
\end{equation*}
Since $P^{**}$ is weak$^*$ to weak$^*$ continuous, $(Pz_\lambda) \to P^{**}x_0^{**}$ in the weak$^*$ topology. Thus, 
\[
2 = \vertiii{ P^{**}x_0^{**} } \leq \liminf_\lambda \vertiii{ Pz_\lambda } \leq \limsup_\lambda \vertiii{Pz_\lambda} \leq 2,
\]
so 
\begin{equation*} \label{eq:maximizing-net}
\vertiii{ P z_\lambda } \to 2.
\end{equation*}
For every $\lambda \in \Lambda$, we may write
\begin{equation*} \label{zn-expression}
z_\lambda = l_\lambda + \sum_{i = 1}^{\infty} \gamma_{\lambda,i} (e_0 + \alpha_i e_i)
- \sum_{i = 1}^{\infty} \beta_{\lambda,i} (e_0 + \alpha_i e_i)
\end{equation*}
where $l_\lambda \in L$, the two families $(\gamma_{\lambda,i})$, $(\beta_{\lambda,i})$ of nonnegative coefficients are finitely supported and
\begin{equation} \label{condition-in-zn}
\|l_\lambda\|_1 +
\sum_{i = 1}^{\infty} \gamma_{\lambda,i} +
\sum_{i = 1}^{\infty} \beta_{\lambda,i} \leq 1.
\end{equation}
By \eqref{eq:norm-of-ei} and \eqref{condition-in-zn},
\begin{equation*}
\vertiii{ P z_\lambda } \leq \|l_\lambda\|_1 + \sum_{i = 1}^{\infty} (\gamma_{\lambda,i} + \beta_{\lambda,i}) \alpha_i \leq 2 - \|l_\lambda\|_1 - \sum_{i = 1}^{\infty} (\gamma_{\lambda,i} + \beta_{\lambda,i}) (2 - \alpha_i).
\end{equation*}
Since $\vertiii{Pz_\lambda} \to 2$, for every $i \geq 1$,
\begin{equation} \label{condition-in-ln}
\|l_\lambda\|_1 \to 0, \ \ \ \gamma_{\lambda,i} \to 0 \ \ \ \mbox{and} \ \ \ \beta_{\lambda,i} \to 0.
\end{equation}
Now consider the closed subspace
\begin{equation*}
E := \overline{\spann} \bigl( \{ e_i \}_{i = 1}^{\infty} \bigr) \subseteq X
\end{equation*}
and define
\begin{equation*}
u_\lambda:=  Pz_\lambda -l_\lambda =  \sum_{i=1}^\infty (\gamma_{\lambda,i} - \beta_{\lambda,i}) \alpha_ie_i \in E.
\end{equation*}
Then 
\begin{equation*} \label{eq:ulambda-to-x0}
u_\lambda \xrightarrow{\sigma(X^{**}, X^*)} P^{**}x_0^{**}.
\end{equation*}
Thus, $P^{**}x_0^{**} \in E^{\perp \perp} \equiv E^{**}$. Moreover, 
\begin{equation*} \label{eq:old-norm-bound}
\|u_\lambda\|
\leq \sum_{i = 1}^{\infty}
(\gamma_{\lambda,i} + \beta_{\lambda,i}) \alpha_i \leq 2\sum_{i=1}^\infty (\gamma_{\lambda,i} +\beta_{\lambda,i}) \leq  2,
\end{equation*}
and \eqref{condition-in-ln} yields that 
\begin{equation*} \label{eq:coordinate-limits}
e_j^* \left( u_\lambda \right) = (\gamma_{\lambda,j} - \beta_{\lambda,j}) \alpha_j \to 0 \quad (j\geq 1).
\end{equation*}
It follows that $\|P^{**}x_0^{**}\| \leq \liminf_\lambda \|u_\lambda\|\leq 2$ and $(P^{**}x_0^{**})(e_j^* \restricted_E)=0$ for every $j \geq 1$.

The map 
\[
J_E \colon (E, \|\cdot\|_1) \to (E,\|\cdot\|) \oplus_\infty c_0, \quad J_E (y)= \left( \frac{1}{2}y, (e_i^* (y))_{i=1}^\infty  \right) 
\]
is a well-defined isometry, since the coefficient sequence $(e_i^*(y))_{i=1}^\infty$ belongs to $c_0$ for every $y\in E$. Hence its second adjoint $J_E^{**}$ is an isometry and 
\[
\|y^{**}\|_1 = \max \left\{ \frac{1}{2} \|y^{**}\|, \sup_{i \geq 1} |y^{**}(e_i^* \restricted_E)| \right\} \qquad (y^{**} \in E^{**}).
\]
In particular,
\[
\|P^{**}x_0^{**}\|_1 = \max \left\{ \frac{1}{2} \|P^{**}x_0^{**}\|, \sup_{i \geq 1} |(P^{**}x_0^{**})(e_i^* \restricted_E)| \right\} = \frac{1}{2}\|P^{**}x_0^{**}\| \leq 1.
\]
Since $B_{(E,\|\cdot\|_1)} \subseteq B \subseteq U$, the inclusion $\iota \colon (E,\|\cdot\|_1)\to (X,\vertiii{\cdot})$ satisfies $\|\iota\|\leq 1$. Thus, $\|\iota^{**}\|\leq 1$; hence 
\[
\vertiii{P^{**} x_0^{**}} \leq \|P^{**}x_0^{**}\|_1 \leq 1,
\]
which contradicts $\vertiii{ P^{**} x_0^{**} } = 2$.
\end{proof}

\begin{remark}
    To the best of the authors' knowledge, it is unknown whether there exists an infinite-dimensional Banach space $X$ such that $\QNA(X,X)=\NA(X,X)$ (see \cite[Problem 7.11]{CCJM}). As a consequence of Theorem \ref{Ostrovskii-for-NA1-and-NA2}, every infinite-dimensional Banach space $X$ admits an equivalent norm $\vertiii{\cdot}$ such that 
    \[
    \QNA((X,\vertiii{\cdot}),(X,\vertiii{\cdot}) )\neq \mathcal{L}((X,\vertiii{\cdot}),(X,\vertiii{\cdot})).
    \]
In particular, if $\QNA(X,X)=\mathcal{L}(X,X)$ holds for some infinite-dimensional Banach space $X$, then this equality can be destroyed by a suitable equivalent renorming.
\end{remark}

\noindent 
\textbf{Acknowledgements}: The authors would like to thank Audrey Fovelle and Javier Merí for fruitful conversations on the topic of this manuscript, Gilles Godefroy for kindly answering several questions, and Vladimir Kadets for suggesting Corollary~\ref{cor:finite-head-norming}.

\noindent 
\textbf{Funding}:
S.~Dantas has been supported by the grants PID2021-122126NB-C31 and PID2021-122126NB-C33 funded by MICIU/AEI/ 10.13039/ 501100011033 and by ERDF/EU.
M.~Jung was supported by the research fund of Hanyang University (HY-202500000003346).
M.~Mart\'{\i}n has been supported by the grant PID2021-122126NB-C31 funded by MICIU/AEI/10.13039/501100011033 and ERDF/EU, grant ``Mar\'{\i}a de Maeztu'' Excellence Unit IMAG funded by MICIU/AEI/ 10.13039/ 501100011033 with reference CEX2020-001105-M, and by grant FQM-0185 funded by Junta de Andaluc\'{\i}a.

\end{document}